\documentclass[12pt, reqno]{amsart}

\usepackage{url}
\usepackage{hyperref}
\usepackage{pdfsync}
\usepackage{amssymb}
\usepackage{amsmath}
\usepackage{amsthm}
\usepackage{stmaryrd}
\usepackage{fullpage}	
\usepackage{amscd}   	
\usepackage[all]{xy} 		
\usepackage{comment} 	
\usepackage{mathrsfs}      

\usepackage{enumerate}

\usepackage{tikz}
\usetikzlibrary{decorations.pathmorphing}

\usepackage[alphabetic,backrefs,lite]{amsrefs} 

\usepackage{color}

\newcommand{\marg}[1]{}

\newcommand{\note}[1]{} 				

\newcommand{\defi}[1]{\textsf{#1}} 				

\DeclareMathOperator{\fix}{fix}

\newcommand{\F}{{\mathbb F}}

\newcommand{\Q}{{\mathbb Q}}

\newcommand{\Z}{{\mathbb Z}}

\newcommand{\Qbar}{{\overline{\Q}}}
\newcommand{\Zhat}{{\hat{\Z}}}

\newcommand{\Kbar}{{\overline{K}}}

\newcommand{\calC}{{\mathcal C}}

\newcommand{\calE}{{\mathcal E}}

\newcommand{\calO}{{\mathcal O}}

\newcommand{\scrE}{{\mathscr E}}

\def\Q{\mathbb{Q}}

\def\P{\mathbb{P}}
\def\A{\mathbb{A}}
\def\Z{\mathbb{Z}}

\DeclareMathOperator{\tr}{tr}

\DeclareMathOperator{\rk}{rk}

\DeclareMathOperator{\im}{im}

\DeclareMathOperator{\Jac}{Jac}

\DeclareMathOperator{\Aut}{Aut}
\DeclareMathOperator{\Gal}{Gal}

\DeclareMathOperator{\Spec}{Spec}  
 \DeclareMathOperator{\rank}{rank}

\newcommand{\GL}{\operatorname{GL}}

\newcommand{\SL}{\operatorname{SL}}

 \numberwithin{equation}{subsection}
\newtheorem{theorem}[subsection]{Theorem}
\newtheorem{lemma}[subsection]{Lemma}
\newtheorem{corollary}[subsection]{Corollary}

\theoremstyle{definition}
\newtheorem{definition}[subsection]{Definition}
\newtheorem{question}[subsection]{Question}

\newtheorem{example}[subsection]{Example}

\theoremstyle{remark}
\newtheorem{remark}[subsection]{Remark}

 \usepackage{longtable}

\newtheorem*{thmnonum}{Theorem}
\begin{document}

\title[Elliptic curves over $\Q$ and 2-adic Images of Galois]
{Elliptic curves over $\Q$ and 2-adic Images of Galois}
\author{Jeremy Rouse and David Zureick-Brown}

\date{\today}

\maketitle

\begin{abstract}
We give a classification of all possible $2$-adic images of Galois
representations associated to elliptic curves over $\Q$. To this end,
we compute the `arithmetically maximal' tower of 2-power level modular curves, develop techniques to compute their equations, and classify the rational points on these curves. 
\end{abstract}



\vspace{7pt}

\section{Introduction}

Serre proved in \cite{Serre:openImage} that, for an elliptic curve $E$
over a number field $K$ without complex multiplication, the index of
the mod $n$ Galois representation $\rho_{E,n}$ associated to $E$ is
\emph{bounded} -- there is an integer $N_E$ such that for any
$n$, the index of $\rho_{E,n}(G_K)$ in $\GL_2(\Z/n\Z)$ is at most
$N_E$ (equivalently, the mod $\ell$ representation is surjective for
large $\ell$). Serre's proof is ineffective in the sense that it does not
compute $N_E$ explicitly; in fact one conjectures that for $\ell >
37$, $\rho_{E,\ell}$ is surjective.  The early progress on this
problem \cite{Mazur:isogenies} has recently been vastly extended
\cite{biluP:rationalPointsArxiv}, but a proof in the remaining case --
to show that the image cannot be contained in the normalizer of a non-split Cartan -- is elusive and
inaccessible through refinements of Mazur's method.

Mazur's Program B \cite{mazur:rationalPointsOnModular} (given an open
subgroup $H \subset \GL_2(\Zhat)$, classify all elliptic curves $E/K$
such that the image of $\rho_{E} = \varprojlim_n \rho_{E,n}$ is
contained in $H$) suggests a more general uniformity conjecture -- one
expects that for every number field $K$, there exists a constant $B(K)$ such that for every elliptic curve $E/K$ without complex multiplication, the \emph{index} of $\rho_E(G_K)$ in
$\GL_2(\Zhat)$ is bounded by $B(K)$.

Computational evidence supports the uniformity conjecture -- for any
given $E$, \cite{zywina2011surjectivity} gives an algorithm
(implemented in Sage) to compute the set of primes $\ell$ such that
$\rho_{E,\ell}$ is not surjective, and verifies for non-CM $E$ with
$N_E \leq 350000$ that $\rho_{E,\ell}$ is surjective for $\ell >
37$.  Similarly, for small $\ell$ one can compute $\im \rho_{E,\ell}$
directly; \cite{sutherland:imageOfGaloisTalk}
has 
computed $\im \rho_{E,\ell}$ for every elliptic curve in the Cremona
and Stein-Watkins databases for all primes $\ell < 80$. This is a
total of 139 million curves, and Sutherland's results are now listed
in Cremona's tables. In Appendix \ref{App:computing-Image}, we
describe a method using \cite{dokchitsers:frobenius} that can often
provably compute the mod $n$ image of Galois for any elliptic curve.



Complementing this are various results (going as far back as Fricke,
possibly earlier; see \cite[Footnote
1]{mazur:rationalPointsOnModular}) computing equations for  the modular
curve $X_H$ parameterizing $E$ with $\rho_{E}(G_K) \subset H$
(see Section \ref{sec:modular-curves} for a
definition). For instance, \cite{brokerLS:modularViaVolcanoes} have
extended the range of $\ell$ such that one can compute the modular
polynomial $\Phi_{\ell}(X,Y)$ to $\ell \approx 10,000$ and Sutherland
now maintains tables of equations for modular curves (see
e.g.~\cite{sutherland:X1NTables}, \cite{sutherland:constructingPrescribedTorsion}). Recently \cite{dokchitsers:2adic}
(inspired by the earlier 3-adic analogue \cite{elkies:3adic}) computed
equations for the modular curves necessary to compute whether the mod
8, and thus the 2-adic, image of Galois is surjective (i.e.~equations for $X_H$ with reduction $H(8) \subset \GL_2(\Z/8\Z)$ a
maximal subgroup). (See Remark \ref{R:equationsAndRationalPoints} for more such examples.)

In many cases these equations have been used to compute the rational
points on the corresponding curves; see Remark
\ref{R:equationsAndRationalPoints} for some examples.  Applications
abound. In addition to verifying low level cases of known
classification theorems such as \cite{Mazur:isogenies} (in this spirit
we note the outstanding case of the ``cursed'' genus 3 curve
$X_{ns}^{+}(13)$ \cite[Remark 4.10]{biluP:rationalPointsArxiv},
\cite{baran2011exceptional}) and verifying special cases of the
uniformity problem, various authors have used the link between
integral points on modular curves and the class number one problem to
give new solutions to the class number one problem; see
\cite[A.5]{Serre1997}, and more recently
\cite{baran:normalizersClassNumber}, \cite{Baran:level9},
\cite{chen:classNumberOne}, \cite{schoofT:CNOneIntegral11}, and
\cite{kenku:classNumberOneOntegral7}.


\subsection*{Main theorem}

In the spirit of Mazur's `Program B', we consider a ``vertical'' variant of the uniformity problem. For any
prime $\ell$ and number field $K$, it follows from Falting's Theorem
and a short argument
 (e.g.~\cite[Theorem 1.2]{Arai:uniformLowerBounds} plus Goursat’s Lemma)
that there is a bound $N_{\ell, K}$ on the index of the image of the
$\ell$-adic representation associated to any elliptic curve over
$K$. The uniformity conjecture implies that for $\ell > 37$,
$N_{\ell,\Q} = 1$, but $N_{\ell}$ can of course be larger for $\ell \leq
37$. Actually even more is true -- the uniformity conjecture would
imply the existence of a universal constant $N$ bounding the index of
$\rho_{E,n}(G_{\Q})$ for every $n$ (equivalently, bounding the index
of $\rho_{E}(G_{\Q}$); see \cite[Theorem 1.1]{zywina:openImageBounds}).
\\

In this spirit, we give a complete classification of the possible $2$-adic images of Galois representations associated to non-CM elliptic curves over $\Q$ and, in particular, compute $N_{2,\Q}$.

\begin{theorem}
\label{T:mainTheorem}
Let $H \subseteq \GL_{2}(\Z_{2})$ be a subgroup, and $E$ be an elliptic curve whose $2$-adic image is contained in $H$. Then one of the following holds:
\begin{itemize}
\item The modular curve $X_{H}$ has infinitely many rational points.
\item The curve $E$ has complex multiplication.
\item The $j$-invariant of $E$ appears in the following table \ref{tab:mainTheorem} below.
\end{itemize}
\end{theorem}

\begin{figure}
\begin{tabular}{c|cc}
$j$-invariant & level of $H$ & Generators of image\\
\hline
$2^{11}$ & $16$ & 
\parbox[c][7ex]{30ex}{\centering\strut$\begin{bmatrix} 7 & 14 \\ 0 & 1 \end{bmatrix}$, $\begin{bmatrix} 1 & 5 \\ 6 & 11 \end{bmatrix}$, $\begin{bmatrix} 3 & 0 \\ 0 & 7 \end{bmatrix}$\strut} \\
$2^{4} \cdot 17^{3}$ & $16$ & 
\parbox[c][7ex]{30ex}{\centering\strut$\begin{bmatrix} 7 & 0 \\ 0 & 3 \end{bmatrix}$,$\begin{bmatrix} 3 & 5 \\ 14 & 7 \end{bmatrix}$, $\begin{bmatrix} 7 & 7 \\ 2 & 1 \end{bmatrix}$\strut}\\
$\frac{4097^{3}}{2^{4}}$ & $16$ & 
\parbox[c][7ex]{30ex}{\centering\strut$\begin{bmatrix} 3 & 5 \\ 6 & 3 \end{bmatrix}$, $\begin{bmatrix} 3 & 5 \\ 14 & 7 \end{bmatrix}$, $\begin{bmatrix} 7 & 7 \\ 2 & 1 \end{bmatrix}$\strut}\\
$\frac{257^{3}}{2^{8}}$ & $16$ & 
\parbox[c][7ex]{30ex}{\centering\strut$\begin{bmatrix} 7 & 14 \\ 0 & 1 \end{bmatrix}$,$\begin{bmatrix} 5 & 0 \\ 0 & 1 \end{bmatrix}$, $\begin{bmatrix} 1 & 5 \\ 6 & 3 \end{bmatrix}$\strut}\\
$-\frac{857985^{3}}{62^{8}}$ & $32$ & 
\parbox[c][7ex]{40ex}{\centering\strut$\begin{bmatrix} 25 & 18 \\ 2 & 7 \end{bmatrix}$,$\begin{bmatrix} 25 & 25 \\ 2 & 7 \end{bmatrix}$, $\begin{bmatrix} 1 & 0 \\ 8 & 1 \end{bmatrix}$,
$\begin{bmatrix} 25 & 11 \\ 2 & 7 \end{bmatrix}$\strut}\\
$\frac{919425^{3}}{496^{4}}$ & $32$ & 
\parbox[c][7ex]{40ex}{\centering\strut$\begin{bmatrix} 29 & 0 \\ 4 & 1 \end{bmatrix}$,$\begin{bmatrix} 31 & 27 \\ 0 & 1 \end{bmatrix}$, $\begin{bmatrix} 1 & 4 \\ 0 & 1 \end{bmatrix}$,
$\begin{bmatrix} 31 & 31 \\ 2 & 1 \end{bmatrix}$\strut}\\
$-\frac{3 \cdot 18249920^{3}}{17^{16}}$ & $16$ & 
\parbox[c][7ex]{30ex}{\centering\strut$\begin{bmatrix} 4 & 7 \\ 15 & 12 \end{bmatrix}$,$\begin{bmatrix} 7 & 14 \\ 7 & 9 \end{bmatrix}$, $\begin{bmatrix} 2 & 1 \\ 11 & 9 \end{bmatrix}$\strut}\\
$-\frac{7 \cdot 1723187806080^{3}}{79^{16}}$ & $16$ & 
\parbox[c][7ex]{30ex}{\centering\strut$\begin{bmatrix} 4 & 7 \\ 15 & 12 \end{bmatrix}$,$\begin{bmatrix} 7 & 14 \\ 7 & 9 \end{bmatrix}$, $\begin{bmatrix} 2 & 1 \\ 11 & 9 \end{bmatrix}$\strut}
\end{tabular}
  \caption{Exceptional j-invariants from Theorem \ref{T:mainTheorem} }
\label{tab:mainTheorem}
\end{figure}

\begin{remark}
The level of a subgroup $H$ is the smallest integer $2^{k}$ so that $H$ contains
all matrices $M \equiv I \pmod{2^{k}}$. Also, we consider action of the matrices in $\GL_{2}(\Z_{2})$
on the \emph{right}. That is, we represent elements of $E[2^{k}]$ as row vectors $\vec{x}$,
and the image of Galois on an element of $E[2^{k}]$ corresponds to $\vec{x} M$.
\end{remark}

\begin{corollary}
\label{C:mainCorollary}
Let $E$ be an elliptic curve over $\Q$ without complex multiplication. Then
  the index of $\rho_{E,2^{\infty}}(G_{\Q})$ divides $64$ or $96$; all such indices occur. Moreover, the image of $\rho_{E,2^{\infty}}(G_{\Q})$ is the inverse image in $\GL_{2}(\Z_{2})$ of the image of $\rho_{E,32}(G_{\Q})$. For non-CM elliptic curves $E/\Q$, there are precisely $1208$ possible images for $\rho_{E,2^{\infty}}$.
  \end{corollary}

\begin{remark}
\label{exceptionalpoints}
All indices dividing $96$ occur for infinitely many elliptic curves. For the first six $j$-invariants in the
table above, the index of the image is $96$, and for these, $-I \in H$ and this index occurs for all quadratic twists. Additionally, there are several subgroups $H$ with $-I \not \in H$ and $X_H \cong \P^1$, so that the there are infinitely many $j$-invariants such that the index is $96$.
Index 64 only occurs for the last two $j$-invariants in the above table,
which occur as the two non-cuspidal non-CM rational points on the genus 2 curve $X_{ns}^+(16)$ ($X_{441}$ on our list; see the analysis of Subsection \ref{ssec:rank2}), which classifies $E$ whose mod 16 image is contained in the normalizer of a non-split Cartan. (The second $j$-invariant was missed in \cite{baran:normalizersClassNumber}, because the map from $X_{ns}^{+}(16)$ to the $j$-line was not correctly computed. In this computation, Baran relied on earlier computations
of Heegner, and the error could be due to either of them.) The smallest conductor of an elliptic curve with this second $j$-invariant is $7^2\cdot79\cdot106123^2$ (which is greater than $4\cdot 10^{13}$).
\end{remark}

\begin{remark}
An application of the classification is an answer to the following
question of Stevenhagen: when can one have $\Q(E[2^{n+1}]) = \Q(E[2^{n}])$ for
a non-CM curve $E$? The answer is that if $n > 1$, $\Q(E[2^{n+1}])$ is larger
than $\Q(E[2^{n}])$. On the other hand, there is a one-parameter family
of curves for which $\Q(E[2]) = \Q(E[4])$. These are parametrized by
the modular curve $X_{20b}$, and one example is the curve
$E: y^{2} + xy + y = x^{3} - x^{2} + 4x - 1$.
\end{remark}

\begin{remark}
The classification above plays a role in Gonz\'alez-Jim\'enez and Lozano-Robledo's classification of all
cases in which $\Q(E[n])/\Q$ is an abelian extension of $\Q$. See \cite{GJLR}.
\end{remark}

  \begin{remark}[Failure of Hilbert irreducibility for a non-rational base]
    A surprising fact is that not every subgroup $H$ such that $X_H(\Q)$ is infinite occurs as the image of Galois of an elliptic curve over $\Q$; see Section \ref{sec:curious-example}.
  \end{remark}

  \begin{remark}[Related work]

In preparation by other authors is a related result  \cite{sutherlandZ:imageOfGalois} -- for every subgroup $H \subset \GL_2(\Z_{\ell})$ such that $-I \in H$, $\det(H) =\widehat{\Z}^{\times}$, and $X_H$ has
  genus 0, they compute equations for $X_H$, whether $X_H(\Q) = \emptyset$ and, if not, equations for the map $X_H \to X(1)$.
  \end{remark}

  \begin{remark}[Connection to arithmetic dynamics]

The image of the $2$-adic representation is connected with the following
problem in arithmetic dynamics. Given an elliptic $E/\Q$ and
a point $\alpha \in E(\Q)$ of infinite order, what is the density
of primes $p$ for which the order of the reduction $\tilde{\alpha} \in E(\F_{p})$ is odd?

In \cite{jonesRouse:galoisTheoryOf}, Rafe Jones and the first author study this
question, and show (see \cite{jonesRouse:galoisTheoryOf}*{Theorem 3.8}) that if for each $n$, $\beta_{n}$
is a chosen preimage of $\alpha$ with $2^{n} \beta_{n} = \alpha$ and
the fields $\Q(\beta_{n})$ and $\Q(E[2^{n}])$ are linearly disjoint
for all $n$, then this density is given by
\[
  \int_{\im \rho_{E,2^{\infty}}} |\det(M-I)|_{\ell} \, d\mu,
\]
an integral over the $2$-adic image. In the case that
$\rho_{E,2^{\infty}}$ is surjective, this density equals
$\frac{11}{21} \approx 0.5238$.  Our calculations show that for a
non-CM elliptic curve $E$, this generic density can be as large as
$\frac{121}{168} \approx 0.7202$ (corresponding to elliptic curves
with no rational $2$-torsion, square discriminant, whose mod $4$ image
does not contain $-I$, namely curves parametrized by $X_{2a}$), and
as small as $\frac{1}{28} \approx 0.0357$ which is attained for
several $2$-adic images, including elliptic curves whose torsion
subgroup is $\Z/2\Z \times \Z/8\Z$. The generic density is listed
on the summary page for each subgroup.
  \end{remark}

We now give a brief outline of the proof of Theorem~\ref{T:mainTheorem}.
For a subgroup $H$ of $\GL_2(\Z_2)$ of finite index, there is some $k$ such that $\Gamma(2^k) \subset H$. The non-cuspidal points of the modular curve $X_H := X(2^k) / H$ then roughly classify elliptic curves whose $2$-adic image of Galois is contained in $H$; see Section \ref{sec:modular-curves} for a more precise definition.\\

The idea of this paper is to find all of the rational points on the
``tower'' of 2-power level modular curves (see Figure \ref{fig:tower}).
We only consider subgroups $H$ such that  $H$ has
surjective determinant and contains an element with determinant $-1$
and trace zero (these conditions are necessary for $X_H(\Q)$  to be non-empty). In our proof, we will handle the case $-I \in H$ first; see Subsection \ref{ssec:background-universal-curves} for a discussion of $X_H$ and the distinction between the cases $-I \in H$ and $-I \not \in H$.

\begin{proof}[Proof of Theorem \ref{T:mainTheorem}]
  The proof naturally breaks into the following steps.
\begin{enumerate}
\item (Section \ref{sec:group-theory}.) First we compute a collection $\calC$ of open subgroups $H \subset \GL_2(\Z_2)$ such that every open $K \subset \GL_2(\Z_2)$ which satisfies the above necessary conditions and which is not in $\calC$ is contained in some $H \in \calC$ such that $X_H(\Q)$ is finite. (See Figure \ref{fig:tower} for those with $-I \in H$.)

\item (Section \ref{sec:comp-equat-x_h}.) Next, we compute, for each $H \in \calC$ equations for (the coarse space of) $X_H$ and, for any $K$ such that $H \subset K$,  the corresponding map $X_H \to X_{K}$.
\item (Section \ref{sec:universal-curve}.) Then, for $H \in \calC$ such that $-I \not \in H$ we compute equations for the universal curve $E \to U$, where $U \subset X_H$ is the locus of points with $j \ne 0, 1728$ or $\infty$.
\item (Remainder of paper.)
 Finally, with the equations in hand, we determine $X_H(\Q)$ for each $H \in \calC$. The genus of $X_{H}$ can be as large as $7$.
\item (Appendix.) If we find a non-cuspidal, non-CM rational point on a curve
$X_{H}$ with genus $\geq 2$, we use computations of resolvent polynomials
(as described in \cite{dokchitsers:frobenius})
to prove that the $2$-adic image for the corresponding elliptic curve $E$
is $H$.
\end{enumerate}
\end{proof}

\begin{remark}[\'Etale descent via group theory]
The analysis of rational points on the collection of $X_H$ involves a variety of techniques, including
local methods, Chabauty and elliptic curve Chabauty, and \'etale descent.

To determine the rational points on some of the genus 5 and 7 curves we invoke a particularly novel (and to our knowledge new) argument, combining \'etale descent with group theory. In short, some of the $X_H$ admit an \'etale double cover $Y \to X_H$ such that $Y$ is isomorphic to $X_{H'}$ for some subgroup $H'$ of $H$. More coincidentally, each of the twists $Y_d$ relevant to the \'etale descent are \emph{also} isomorphic to modular curves $X_{H'_d}$ for some group $H'_d$. And finally, each group $H'_d$ is a subgroup of some additional larger group $H''_d$ such that $X_{H''_d}$ is a curve with finitely many rational points we already understand (e.g.~a rank 0 elliptic curve), and the map $X_{H'_d} \to X_{H''_d}$ determines $X_{H'_d}(\Q)$ and thus, by \'etale descent $X_{H}(\Q)$. This method is applicable to 20 out of the 24 curves of genus greater than $3$ that we must consider. See Subsection \ref{ssec:etale-descent-via}.
\end{remark}

\subsection*{Acknowledgements}

We thank Jeff Achter, Nils Bruin, Tim Dokchitser, Bjorn Poonen, William Stein, Michael
Stoll, Drew Sutherland, and David Zywina for useful
conversations and University of Wisconsin-Madison's Spring 2011 CURL
(Collaborative undergraduate research Labs) students (Eugene Yoong,
Collin Smith, Dylan Blanchard) for doing initial group theoretical
computations. The second author is supported by an NSA Young
Investigator grant. We would also like to thank anonymous referees for helpful comments and suggestions
that have improved the paper.

\section{The Modular curves $X_H$}
\label{sec:modular-curves}

Given a basis $(P_1,P_2)$ of $E(\Qbar)[N]$ we identify $\psi \colon (\Z/N\Z)^2 \cong E(\Qbar)[N] $ via the map $\psi(e_i) = P_i$. This gives rise to \emph{two} isomorphisms $\iota_1,\iota_2\colon\Aut E(\Qbar)[N] \cong \GL_2(\Z/N\Z)$ (corresponding to a choice of left vs right actions) as follows: if $\phi\in\Aut E(\Qbar)[N]$ satisfies
\begin{align*}
 \phi(P_1)  = &\,  aP_1 + cP_2 \\
 \phi(P_2)  = &\, bP_1 + dP_2,
\end{align*}
then we define
\[
\iota_1(\phi) := \left[ \begin{matrix} a & b \\ c & d \end{matrix} \right]\, \mbox{and } \iota_2(\phi) := \left[ \begin{matrix} a & c \\ b & d \end{matrix} \right].
\]
These correspond respectively to left (via column vectors) and right (via row vectors) actions of $\GL_2(\Z/N\Z)$ on $(\Z/N\Z)^2$. Alternatively, $\iota_i(\phi)$ is defined by commutativity of the diagram
\[
\xymatrix{
(\Z/N\Z)^2 \ar[r]^{\psi} \ar[d]_{\iota_i(\phi)} & E[N] \ar[d]^{\phi} \\
(\Z/N\Z)^2 \ar[r]_{\psi} &E[N]
 }
\]
where we consider $\iota_i(\phi)$ acting on the left (via column vectors) for $i = 1$ and on the right (via row vectors) for $i = 2$.

Throughout this paper we use only right actions, and in particular define $\rho_{E,N}\colon G_K \to
\GL_2(\Z/N\Z)$ as  $\rho_{E,N}(\sigma) := \iota_2(\sigma)$ (this is consistent with, for instance, \cite{Shimura:automorphicFunctions}). Many sources are ambiguous about this choice, but the ambiguity usually does not
matter (see Remark \ref{R:left-vs-right}).

For an integer $N$, we define the modular curve $Y(N)/\Q$ to be the moduli space parameterizing pairs $(E/S, \iota)$, where $E$ is an elliptic curve over some base scheme $S/\Q$ and $\iota$ is an isomorphism $M_S := (\Z/N\Z)_{S}^2 \cong E[N]$, and define $X(N)$ to be its smooth compactification (see \cite[II]{DeligneRapoport} for a modular interpretation of the cusps). Note that $X(N)$ is not geometrically connected (and thus differs from the geometrically connected variant of \cite[Section 2]{mazur:rationalPointsOnModular} where $\iota$ is ``canonical'' in that it respects the Weil pairing), and that a matrix $A \in \GL_2(\Z/N\Z)$ acts on $X(N)$ (on the right) via precomposition with
\[
M \xrightarrow{A} M,\, \vec{v} \mapsto \vec{v}A
\]
so that $A\cdot (E,\iota) := (E,\iota\circ A)$.

Following \cite{DeligneRapoport}, for a subgroup $H$ of $\GL_2(\Zhat)$ and an integer $N$ such that $H$ contains the kernel of the reduction map $\GL_2(\Zhat) \to \GL_2(\Z/N\Z)$, we define $X_H$ to be the quotient of the modular curve  $X(N)$ by the image $H(N)$ of $H$ in $\GL_2(\Z/N\Z)$.  This quotient is independent of $N$, is geometrically connected if $\det(H) =\widehat{\Z}^{\times}$,  and roughly classifies elliptic curves whose adelic image of Galois is contained in $H$.  By the definition of $X_H$ as a quotient,   the non-cuspidal $K$-rational points of $X_H$ correspond to $G_K$-stable $H$-orbits of pairs $(E,\iota)$; we make the translation to the image of Galois more precise in the following lemma.

\begin{lemma} Let  $E$ be an elliptic curve over a number field $K$. Then there exists an $\iota$ such that $(E,\iota) \in X_H(K)$ if and only if $\im \rho_{E,n}$ is contained in a subgroup conjugate to  $H$.
\end{lemma}

\begin{proof}

For  $\sigma \in G_K$ and $(E,\iota)  \in X_H(K)$, $\iota^{\sigma}$ is defined to be the composition $M_K \xrightarrow{\iota} E[N] \xrightarrow{\sigma} E[N]$.
If $(E,\iota)  \in X_H(K)$, then for every $\sigma \in G_K$, there is some $A \in H$ such that  $\iota^{\sigma} = \iota\circ A$.
Set $P_i := \iota(e_i)$ and suppose that $A = \left[ \begin{matrix} a & c \\ b & d \end{matrix} \right]$.
Then
\begin{align*}
(\iota \circ A)(e_1) =\iota(e_1A) = \iota(ae_1 + ce_2) & = aP_1 + cP_2 =  P_1^{\sigma} \\
(\iota \circ A)(e_2) =\iota(e_2A) = \iota(be_1 + de_2) & = bP_1 + dP_2 =  P_2^{\sigma},
\end{align*}
so $\rho_{E,N}(\sigma) = A$ and $\im \rho_{E,N} \subset H$ as claimed.

Conversely, let $P_1,P_2$ be a basis of $E(\overline{K})[N]$ such that $\im \rho_{E,N} \subset H$ with respect to this basis, and define $\iota$ by $\iota(e_i) := P_i$.
For $\sigma \in G_K$, $\rho_{E,N}(\sigma) = A$ where
\[ A := \left[ \begin{matrix} a & c \\ b & d \end{matrix}, \right]\]
and
\begin{align*}
P_1^{\sigma}  &:= aP_1 + cP_2   \\
P_2^{\sigma}  &:= bP_1 + dP_2.
\end{align*}
By assumption, $A \in H$; moreover
\begin{align*}
\iota^{\sigma}(e_1) = \iota(ae_1 + ce_2) & = aP_1 + cP_2  \\
\iota^\sigma(e_2) = \iota(be_1 + de_2) & = bP_1 + dP_2
\end{align*}
and so $\iota^{\sigma} = \iota\circ A$, which proves the converse.
\end{proof}

\begin{remark}
\label{R:left-vs-right}

 As discussed above, a choice of basis for $E[N]$ gives rise to \emph{two} isomorphisms $\iota_1,\iota_2\colon\Aut E[N] \cong \GL_2(\Z/N\Z)$, via column and row vectors. Given $K \subset \Aut E[N]$, the images $\iota_i(K)$ generally differ; in fact, $\iota_1(K) = \iota_2(K)^T$, where we define the \emph{transpose} $H^T := \{ A^T : A \in H\}$.
In the literature the choice of left or right action is often ambiguous, but usually does not matter: for many common $H$ (e.g.~the normalizer of a Cartan subgroup) $H$ is conjugate to $H^T$ and the modular curves $X_H$ and $X_{H^T}$ are thus isomorphic. This is an issue in this paper; if instead we use $\iota_1$, then $X_H$ parametrizes $E$ with image contained in $H^T$ rather than $H$, and in general $H^T$ and $H$ are not conjugate.
\end{remark}

In general  $X_H$ is a stack, and if $-I \in H$, then the stabilizer of every point contains $\Z/2\Z$. (Some, but not all, of the CM points with $j =0$ or $12^3$ will have larger stabilizers.) In contrast, when $-I \not \in H$, $X_H$ no longer has a generic stabilizer, but is generally still a stack since the CM points may have stabilizers. When $-I \in H$, quadratic twisting preserves the property that $\im \rho_{E,N} \subset H$; in contrast, when $-I \not \in H$, given a non-CM  elliptic curve $E/K$ such that $j(E)$ is in the image of the map $j\colon X_H(K) \to \P^{1}(K)$, there is a unique quadratic twist $E_d$ of $E$ such that $\im \rho_{E,N} \subset H$ (see Lemma \ref{L:uniqueness-of-twisting} below).
\vspace{3pt}

There exists a \defi{coarse space morphism}, i.e.~a morphism $\pi\colon X_H \to X$, where $X$ is a scheme, such the map $X_H(\Qbar) \to X(\Qbar)$ is a bijection, and any map from $X_H$ to a scheme uniquely factors through this morphism.
We compute equations for the coarse space of $X_H$ (and with no confusion will use the same notation $X_H$ for the coarse space). The coarse space has the following moduli interpretation --  given a number field $K$ and a $K$-point $t$ of the coarse space, there exists an elliptic curve with $j$-invariant $j(t)$ (where $j$ is the map $X \to X(1)$) satisfying $\im \rho_{E,N} \subset H$, and conversely, for any $E/K$ such that $\im \rho_{E,n} \subset H$, there exists a $K$-point $t$ of the coarse space of $X_H$ such that $j(t) = j(E)$.
\\

For more details see \cite[IV-3]{DeligneRapoport}; alternatively, for
a shorter discussion see \cite[Section
3]{baran:normalizersClassNumber}, \cite[A.5]{Serre1997}, or
\cite[Section 2]{mazur:rationalPointsOnModular}.

\subsection{Universal curves}
\label{ssec:background-universal-curves}

Suppose that $-I \not \in H$. Since we are not interested in the CM points anyway, we consider the complement $U \subset X_H$ of the cusps and preimages
on $X_{H}$ of $j = 0$ and $j = 1728$. Then $U$ is a scheme, so there exists a universal curve $\scrE \to U$; i.e.~a surface $\scrE$ with a map $\scrE \to U$ such that for every $t \in U(K)$, the fiber $\calE_t$ is an elliptic curve over $K$ without CM such that $\im \rho_{E,n} \subset H$, and conversely for any elliptic curve $E$ over a field $K$ such that $\im \rho_{E,n} \subset H$ there exists a (non-unique) $t \in U(K)$ such that the $E \cong \calE_t$.
\vspace{3pt}

In preparation for Section \ref{sec:universal-curve} (where we compute equations for $\scrE \to U$), we prove a preliminary lemma on the shape of the defining equations of $\scrE$.
\begin{lemma}
\label{L:universal-weierstrass-embedding}
Let $f \colon \scrE \to U$ be as above and assume that $U \subset \A^1$. Then there exists a closed immersion $\scrE \hookrightarrow \P^2_U$ given by a homogeneous polynomial
\[
Y^2Z - X^3 - aXZ^2 - bZ^3
\]
 where $a,b \in \Z[t]$.
\end{lemma}

\begin{proof}
  The identity section $e\colon U \to \scrE$ is a closed immersion whose image $e(U)$ is thus a divisor on $\scrE$ isomorphic to $U$. By Riemann-Roch, the fibers of the pushforward $f_*\calO(3e(U))$ are all 3-dimensional, so by the theorem on cohomology and base change $f_*\calO(3e(U))$ is a rank 3 vector bundle on $U$. Since $U \subset \A^1$, $U$ has no non-trivial vector bundles and so $f_*\calO(3e(U))$ is trivial. Let $\calO_U^{\oplus 3} \cong f_*\calO(3e(U))$ be a trivialization given by sections $1,x,y$, where 1 is the constant section 1 (given by adjunction), $x$ has order 2 along $e(U)$, and $y$ has order 3. These sections determine a surjection $f^*f_*\calO(3e(U)) \to \calO(3e(U))$ and thus a morphism $\calE \to \P^2_U$ which, since the fibers over $U$ are closed immersions, is also a closed immersion; $1,x,y$ satisfy a cubic equation (this is true over the generic point, so  true globally) and, since we are working in characteristic 0, can be simplified to short Weierstrass form as desired.
\end{proof}

\section{Subgroups of $\GL_2(\Z_2)$}
\label{sec:group-theory}

\begin{definition}
  Define a subgroup $H \subset \GL_2(\Z_2)$ to be \defi{arithmetically
    maximal} if
\begin{enumerate}
\item $\det \colon H \to \Z_{2}^{\times}$ is surjective,
\item there is an $M \in H$ with determinant $-1$ and trace zero, and
\item there is no subgroup $K$ satisfying (1) and (2) with $H
  \subseteq K$ so that $X_{K}$ has genus $\geq 2$.
\end{enumerate}
\end{definition}

If $E/\Q$ is an elliptic curve and $H = \rho_{E,2^{\infty}}(G_{\Q})$, then
the properties of the Weil pairing prove that $\det \colon H \to \Z_{2}^{\times}$
is surjective. Also, the image of complex conjugation in $H$ must be a matrix
$M$ with $M^{2} = I$ and $\det(M) = -1$. This implies that the trace of $M$
equals zero.

\begin{remark}
  After the subgroup and model computations were complete, David
  Zywina and Andrew Sutherland pointed out that if $E/\Q$ is an
  elliptic curve, complex conjugation fixes an element of $E[n]$. This
  gives further conditions on a matrix $M$ that could be the image of complex
conjugation, and rules out a handful of other subgroups.
\end{remark}

We enumerate all of the arithmetically maximal subgroups of
$\GL_{2}(\Z_{2})$ by initializing a queue containing only
$H = \GL_{2}(\Z_{2})$. We then remove a subgroup $H$ from the queue,
compute all of the open maximal subgroups $M \subseteq H$. We add $M$
to our list of potential subgroups if (i) $\det \colon M \to
\Z_{2}^{\times}$ is surjective, (ii) $-I \in M$, (iii) $M$ contains a
matrix with determinant $-1$ and trace zero, and (iv) if $M$ is not
conjugate in $\GL_{2}(\Z_{2})$ to a subgroup already in our list. If
the genus of $X_{M}$ is zero or one, we also add $M$ to the queue. We
proceed until the queue is empty.

To enumerate the maximal subgroups, we use the following results.
Recall that if $G$ is a profinite group, then $\Phi(G)$, the Frattini
subgroup of $G$, is the intersection of all open maximal subgroups of $G$.
Proposition 2.5.1(c) of \cite{wilson:profiniteGroups} states that
if $K \unlhd G$, $H \subseteq G$ and $K \subseteq \Phi(H)$, then
$K \subseteq \Phi(G)$. Applying this with $H = N \unlhd G$
and $K = \Phi(N)$, we see that $\Phi(N) \subseteq \Phi(G)$.


\begin{lemma}
Suppose that $\Gamma(2^{k}) \subseteq H \subseteq G$ and
$k \geq 2$. If $K$ is a maximal subgroup of $H$, then $\Gamma(2^{k+1}) \subseteq K$.
\end{lemma}
\begin{proof}
We have that $\Gamma(2^{k}) \unlhd H$ and by the above argument, we have
\[
  \Phi(\Gamma(2^{k})) \subseteq \Phi(H).
\]
Now, $\Gamma(2^{k})$ is a pro-2 group and this implies that every open maximal
subgroup of $\Gamma(2^{k})$ has index $2$. Hence,
\[
  \Phi(\Gamma(2^{k})) \supseteq \Gamma(2^{k})^{2}.
\]
If $g \in \Gamma(2^{k})$, $g = I + 2^{k} M$ for some $M \in M_{2}(\Z_{2})$. Then,
\[
  g^{2} = I + 2^{k+1} M + 2^{2k} M^{2} \equiv I + 2^{k+1} M \pmod{2^{k+2}}
\]
provided $k \geq 2$. Hence, the squaring map
gives a surjective homomorphism $\Gamma(2^{k})/\Gamma(2^{k+1}) \to \Gamma(2^{k+1})/\Gamma(2^{k+2})$ for all $k \geq 2$. It follows that an element in $\Gamma(2^{k+1})$ can
be written as a product of squares in every quotient $\Gamma(2^{k})/\Gamma(2^{n+k})$ and since the $\Gamma(2^{n+k})$ form a base for the open neighborhoods of the
identity in $G$, we have that $\Gamma(2^{k+1}) \subseteq \Phi(\Gamma(2^{k}))$.
This yields the desired result.
\end{proof}

The enumeration of the subgroups is accomplished using Magma. The initial
enumeration produces $1619$ conjugacy classes of subgroups. The computation
of the lattice of such subgroups finds that many of these are contained in
subgroups $H$ where the genus of $X_{H}$ is $\geq 2$. These are then removed,
resulting in $727$ arithmetically maximal subgroups. The arithmetically
maximal subgroups can have genus as large as $7$ and index as large as $192$.

\begin{figure}
  \includegraphics[width=\textwidth,height=\textheight,keepaspectratio]{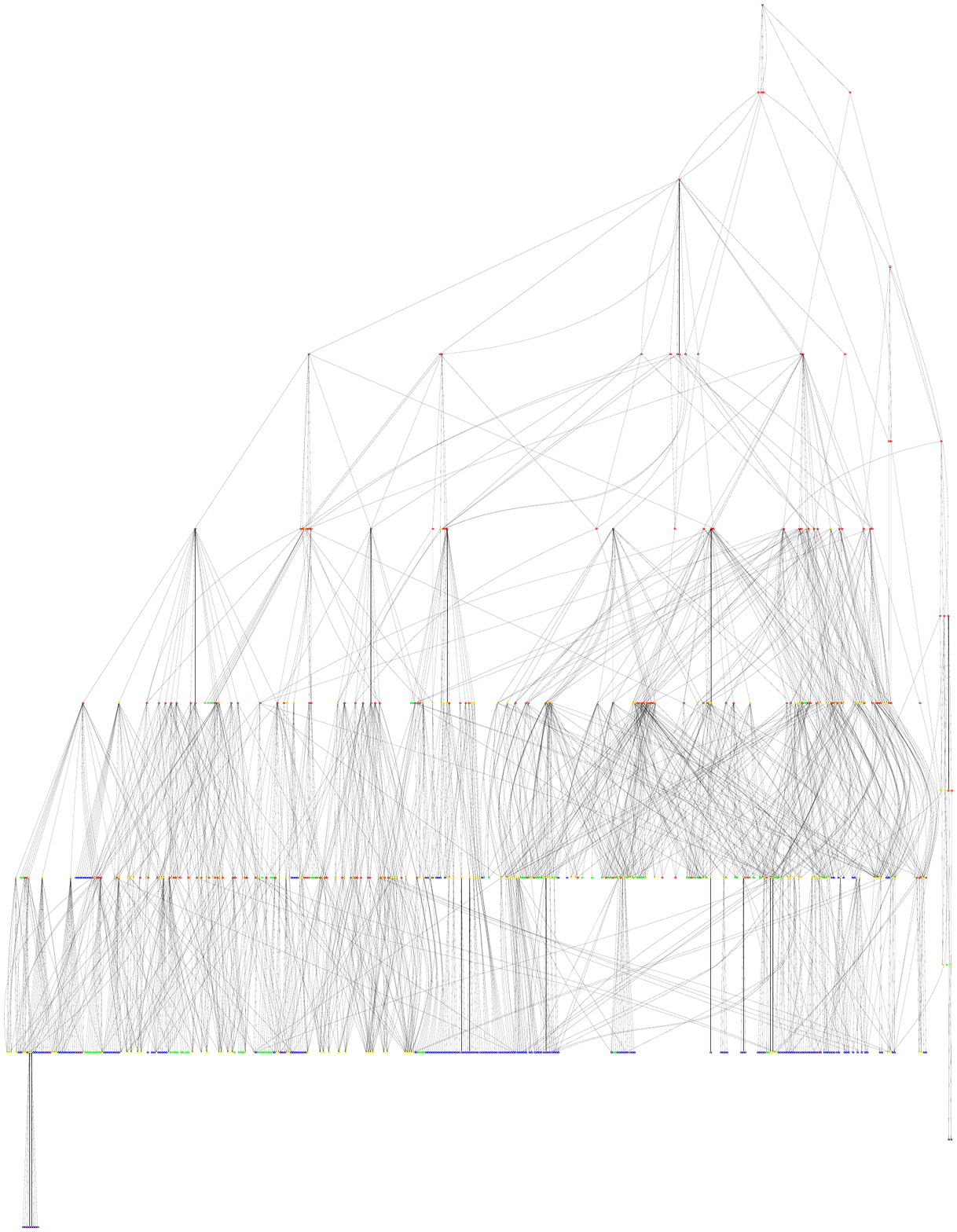}
  \caption{The tower of arithmetically maximal subgroups $H \subset \GL_2(\Z_2)$ with $-I \in H$.}
\label{fig:tower}
\end{figure}

\section{Computing equations for $X_H$ with $-I \in H$}
\label{sec:comp-equat-x_h}

Here we discuss the computation of equations for $X_H$ as $H$ ranges over the arithmetically maximal subgroups of $\GL_2(\Z_2)$.

\begin{remark}
\label{R:equationsAndRationalPoints}
  Equations for some of these curves already appear in the literature; see
\cite{sutherland:X1NTables}, \cite{sutherland:constructingPrescribedTorsion},
\cite{GJG:modularGenus2},
\cite{heegner:diophantische},
\cite[Table 12.1]{knapp:ellipticCurvesBook},
\cite{shimura:equationsOfModularCurves},
\cite{dokchitsers:2adic},
\cite[Proof of Lemma 3.2]{momose:rationalSplit},
\cite{baran:normalizersClassNumber},
\cite{heegner:diophantische}
\cite{mcmurdy:tables},
\cite{zywina2011surjectivity}*{3.2}
 for equations of
$X_0(N)$ for $N=2,  4, 8, 16, 32, 64$, $X_1(N)$ for $N=2,  4, 8, 16$, $X_H$ with $H \subset \GL_2(\Z/8\Z)$ maximal, $X^+_{ns}(N)$ for $N=2,  4, 8,16$, and various other small genus modular curves.

\end{remark}

We first assume that $-I \in H$. Let $H_{n}$ be the $n$th subgroup in our list of $727$
(as given in the file {\tt gl2data.txt}), and let $X_{n} = X_{H_{n}}$.
Instead of constructing the coverings $X_{n} \to X_{1}$ directly, we will instead
construct coverings $X_{n} \to X_{m}$ so that $H_{n}$ is a maximal subgroup of $H_{m}$ and compose to get $X_n \to X_1$. In almost all cases the degree of the covering $X_{n} \to X_{m}$ is $2$.
(The exceptions are $X_{6} \to X_{1}$, which has degree $3$,
and $X_{7} \to X_{1}$, $X_{55} \to X_{7}$, and $X_{441} \to X_{55}$ which all have
degree $4$. The curves $X_{1}$, $X_{7}$, $X_{55}$ and $X_{441}$ are
the curves $X_{ns}^{+}(2^{k})$ for $1 \leq k \leq 4$.)

In this process, if we find that $X_{n}$ is a pointless conic,
a pointless genus one curve, or an elliptic curve of rank zero, we do not
compute any further coverings of $X_{n}$. For this reason, it is only
necessary for us to compute models of $X_{n}$ for $345$ choices of $n$.

In Section 6.2 of \cite{Shimura:automorphicFunctions}, Shimura shows that
the field $L$ of modular functions on $X(N)$ whose Fourier coefficients
at the cusp at infinity are contained in $\Q(\zeta_{N})$ is generated by
\[
  f_{\vec{a}}(z) = \frac{9}{\pi^{2}} \frac{E_{4}(z) E_{6}(z)}{\Delta(z)} \wp_{z}\left(\frac{cz+d}{N}\right)
\]
where $\vec{a} = (c,d)$ and $(c,d) \in (\Z/N\Z)^{2}$ has order $N$. Here,
$\wp_{z}(\tau)$ is the classical Weierstrass $\wp$-function
attached to the lattice $\langle 1,z \rangle$.

Theorem 6.6 of \cite{Shimura:automorphicFunctions} shows that the
action of $\GL_{2}(\Z/N\Z)$ given by $f_{\vec{a}} | M = f_{\vec{a} M}$
uniquely extends to the entire field $L$ and is an automorphism of $L$
fixing $\Q(j)$. Moreover, $\Gal(K/\Q(j)) \cong \GL_{2}(\Z/N\Z) / \{ \pm I \}$,
and $\zeta_{N} | M = \zeta_{N}^{\det M}$. When $M \in \SL_{2}(\Z/N\Z)$,
the action of $M$ on $K$ is $\Q(\zeta_{N})$-linear and agrees with the usual
action: if $h \in L$ and $M = \begin{bmatrix} \alpha & \beta \\ \gamma & \delta \end{bmatrix} \in \SL_{2}(\Z/N\Z)$, then
\[
  (h | M)(z) = h\left(\frac{\alpha z + \beta}{\gamma z + \delta}\right).
\]

Given $H \subseteq \GL_{2}(\Z_{2})$ containing $\Gamma(2^{k})$,
we can think of $H$ as a subgroup of $\GL_{2}(\Z/2^{k} \Z)$ (by abuse of notation
also called $H$) using the isomorphism $\Z_{2}/2^{k} \Z_{2} \cong \Z/2^{k} \Z$.
Let $\tilde{H}$ be a subgroup of $\GL_{2}(\Z_{2})$
containing $H$ so that the covering $X_{H} \to X_{\tilde{H}}$ has minimal degree.
Our goal is to find an element $h \in L$ that generates the fixed field of
$H$ over $\Q(X_{\tilde{H}})$, and compute its images under representatives
for the right cosets of $H$ in $\tilde{H}$.

We consider the $\Q(\zeta_{2^{k}})$-subspace $V$ of $L$ spanned by the
functions $f_{\vec{a}}$. It is natural to seek a modular function $h$
in the subspace of $V$ fixed by $H$. However, this approach does not
always succeed.  The map $f_{\vec{a}} \to f_{\vec{a}} \cdot
\frac{\Delta(z)}{E_{4}(z) E_{6}(z)}$ is a bijection between $V$ and
the space of weight $2$ Eisenstein series for $\Gamma(2^{k})$ with
coefficients in $\Q(\zeta_{2^{k}})$ (see \cite{DiamondS:modularForms},
Section 4.6) and the dimension of the space of weight $2$ Eisenstein
series for $H$ is the number of cusps of $X_{H}$ minus one (see
equation (4.3) on page 111 of \cite{DiamondS:modularForms}). If there
is a subgroup $M$ with $H \subseteq M$ for which $X_{H}$ and $X_{M}$
have the same number of cusps, then $V^{H} = V^{M}$ and we will not
succeed in finding a primitive element for $\Q(X_{H})$. Instead, we
will find a subgroup $K \subseteq H$ so that $X_{K}$ has more cusps
than $X_{M}$ for any subgroup $M$ with $K \subseteq M \subseteq H$
(with $K \ne M$). The number of cusps a subgroup $K$ has is the number
of orbits $K \cap \SL_{2}(\Z/2^{k} \Z)$ has in its natural action on
$\mathbb{P}^{1}(\Z/2^{k} \Z)$. If $K \cap \SL_{2}(\Z/2^{k} \Z) =
\Gamma(2^{k})$, the action of $K$ on $\mathbb{P}^{1}(\Z/2^{k} \Z)$
will be trivial, and so $K$ will have more cusps than any larger
subgroup.

Once $K$ is selected, we compute $V^{K \cap \SL_{2}(\Z/2^{k} \Z)}$.
The sum $\sum_{\vec{a}} f_{\vec{a}} \cdot \frac{\Delta(z)}{E_{4}(z)
  E_{6}(z)}$ over all vectors $\vec{a}$ with order $2^{k}$ in
$(\Z/2^{k} \Z)^{2}$ is fixed by $\SL_{2}(\Z/2^{k} \Z)$ and is a
holomorphic modular form of weight $2$. Since there are no nonzero
weight $2$ modular forms for $\SL_{2}(\Z)$, $\sum_{\vec{a}}
f_{\vec{a}} = 0$.  However, as proved by Hecke in \cite{Hecke},
removing any one of these gives a linearly independent set. From this,
we know exactly how $\GL_{2}(\Z/2^{k} \Z)$ acts on the space $V$,
and we can compute subspaces fixed by various subgroups in terms of
a basis, and only compute Fourier expansions when needed.
We use this to compute $V^{K \cap \SL_{2}(\Z/2^{k} \Z)} = \langle w_{1}, w_{2}, \ldots, w_{d} \rangle$ by determining the $\Q(\zeta_{2^{k}})$-subspace of $V$ fixed by generators of $K \cap \SL_{2}(\Z/2^{k} \Z)$. Once this is computed, we determine
$V^{K} = \langle x_{1}, x_{2}, \ldots, x_{m} \rangle$ (a $\Q$-subspace of $V$)
by considering the action of generators of $K$ on $\zeta^{i} w_{j}$.
We select $x = \sum_{i=1}^{m} ix_{i}$ as a ``random'' element of $V^{K}$
and verify that the number of images of $x$ under the action of
$\tilde{H}$ is equal to $[\tilde{H} : K]$.

Finally, we compute the Fourier expansions of the $f_{\vec{a}}$ and
use these to compute the Fourier expansions of the images of $x$.
If $g_{1}, g_{2}, \ldots, g_{r}$ are representatives for the
right cosets of $K$ in $H$, we define
\[
  h = e_{s}(x | g_{1}, x | g_{2}, \ldots, x | g_{r}),
\]
where $e_{s}$ is the degree $s$ elementary symmetric polynomial in $r$
variables. We start with $s = 1$ and check if there are
$[\tilde{H} : H]$ images of $h$ under the action of the right cosets of $H$
in $\tilde{H}$. We increment $s$ until this occurs (and find that
in all cases we can take $s \leq 3$). We build the polynomial
\[
  F(t) = \prod_{g \in T} (t - h|g).
\]
Each of the coefficients of $F(t)$ is an element of $\Q(X_{\tilde{H}})$, which
can be recognized from their Fourier expansion. In the case that $X_{\tilde{H}}$ has genus one, we use the following result, whose proof is straightforward and we omit.

\begin{lemma}
\label{ellipticfunc}
Let $E \colon y^{2} + a_{1} xy + a_{3} y = x^{3} + a_{2} x^{2} + a_{4} x + a_{6}$ be an
elliptic curve and $g \colon E \to \P^{1}$ be a degree $k$ morphism. Then,
\[
  g = \frac{P(x) + y Q(x)}{R(x)}
\]
where $P$, $Q$ and $R$ are polynomials with $\deg P \leq 3k-3$,
$\deg Q \leq 3k-5$ and $\deg R \leq 3k-3$.
\end{lemma}

We then have explicitly
that $\Q(X_{H}) = \Q(X_{\tilde{H}})[t]/( F(t) )$. At this point we use
some straightforward techniques to simplify the model generated.

\begin{example}
We will consider the example of the covering $X_{57} \to X_{22}$. The subgroup
$H_{22}$ is an index $8$, level $8$ subgroup of $\GL_{2}(\Z_{2})$. It is one of three maximal
subgroups (up to $\GL_{2}(\Z_{2})$ conjugacy) of $H_{7}$, which is the unique
maximal subgroup of $\GL_{2}(\Z_{2})$ of index $4$. When the covering $X_{22} \to X_{7}$ was computed, we determined that $X_{22} \cong \P^{1}$ and we computed and
stored the Fourier expansion of a function $f_{22}$ with $\Q(X_{22}) = \Q(f_{22})$. The subgroup $H_{57}$ is an index $2$ subgroup of $H_{22}$, and
$H_{57} \supseteq \Gamma(16)$. It is generated by $\Gamma(16)$ and the matrices
\[
  \left[ \begin{matrix} 11 & 4 \\ 8 & 3 \end{matrix} \right],
  \left[ \begin{matrix} 15 & 11 \\ 0 & 1 \end{matrix} \right],
  \left[ \begin{matrix} 7 & 2 \\ 2 & 1 \end{matrix} \right], \text{ and }
  \left[ \begin{matrix} 15 & 15 \\ 1 & 0 \end{matrix} \right].
\]
Both $H_{22}$ and $H_{57}$ have two cusps. We choose $K$ to be the subgroup generated by $\Gamma(16)$ and the matrices
\[
  \left[ \begin{matrix} 13 & 2 \\ 14 & 11 \end{matrix} \right],
  \left[ \begin{matrix} 1 & 1 \\ 15 & 0 \end{matrix} \right], \text{ and }
  \left[ \begin{matrix} 1 & 0 \\ 7 & 7 \end{matrix} \right].
\]
We have $[H_{57} : K] = 4$. The modular curve $X_{K}$ has $8$ cusps. The
subspace $V$ of $\Q(X_{\Gamma(16)})$ is a $\Q(\zeta_{16})$-vector space
of dimension $95$ spanned by the $f_{\vec{a}}$,
where $\vec{a} = (c,d) \in (\Z/16\Z)^{2}$ and at least one of $c$ or $d$ is odd.
The subspace fixed by $K \cap \SL_{2}(\Z)$ has dimension $7$. Let $g_{1}, g_{2}, \ldots, g_{7}$ be a basis for this space. We consider the
$56$-dimensional $\Q$-vector space spanned by $\{ \zeta_{16}^{i} g_{j} : 0 \leq i \leq 7, 1 \leq j \leq 7 \}$ and we find the $7$-dimensional subspace fixed
by the action of $K$. We select a linear combination of these $7$ functions to
obtain a ``random'' modular function $x(z)$ fixed by $K$.

This $x(z)$ is still represented as a linear combination of the functions $f_{\vec{a}}$. We now compute the $q$-expansions of $x(z) | \gamma$, where $\gamma$
ranges over representatives of the $8$ right cosets of $K$ in $H_{22}$. We
partition these into two sets, 
\[
\{ x_{1}(z), x_{2}(z), x_{3}(z), x_{4}(z) \}
\text{ and } 
\{ x_{1}(z) | \delta, x_{2}(z) | \delta, x_{3}(z) | \delta, x_{4}(z) | \delta \}
\] where the $x_{i}(z)$ are the images of $x(z)$ under cosets of $K$ contained
in $H_{57}$, and $\delta \in H_{22}$ but $\delta \not\in H_{57}$.

We plug the $x_{i}(z)$ into the second elementary symmetric polynomial to
obtain a modular function $h(z)$ for $H_{57}$. Its image $h(z) | \delta$
under the action of $\delta$ is obtained from the $x_{i}(z) | \delta$.
Finally, a generator for $\Q(X_{57})/\Q(X_{22})$ is obtained as a root of
the polynomial
\[
  (x - h(z)) (x - h(z) | \delta).
\]
The function $f_{22}(z)$ with $\Q(X_{22}) = \Q(f_{22})$ has Fourier expansion
\[
  f_{22}(z) = 3 \sqrt{2} + (36 + 24 \sqrt{2})(1+i) q^{1/4} + (288 + 216 \sqrt{2})i q^{1/2} - (480 \sqrt{2} + 720)(1-i) q^{3/4} - 96 \sqrt{2} q + \cdots.
\]
The function $h(z) + h(z) | \delta$ has degree at most $3$,
and in fact we find that
\[
  h(z) + h(z) | \delta = \frac{2^{11} \cdot 3^{3} \cdot (155f_{22}^{2} - 5946f_{22} - 26784)}{f_{22}^{2} + 12f_{22} + 30}.
\]
Similarly, we find that
\[
  (h(z)) (h(z) | \delta) =   \frac{2^{20} \cdot 3^{6} \cdot (174569 f_{22}^{4} - 739788 f_{22}^{3} + 26364168 f_{22}^{2} + 298652832f_{22} + 680985144)}{(f_{22}^{2} + 12f_{22} + 30)^{2}}.
\]
These equations show that there is a modular function $g$ for $X_{57}$
so that $g^{2} = 18 - f_{22}^{2}$. This equation for $X_{57}$ is a conic. Finding
an isomorphism between this conic and $\P^{1}$ yields a function $f_{57}$
for which $\Q(X_{57}) = \Q(f_{57})$. This $f_{57}$ satisfies
\[
  f_{22} = \frac{3f_{57}^{2} + 6f_{57} - 3}{f_{57}^{2} + 1},
\]
which gives the covering map $X_{57} \to X_{22}$. The entire calculation
takes $26$ seconds on a 64-bit 3.2 GHz Intel Xeon W3565 processor.

Taking, for example,
$f_{57} = 0$ gives $f_{22} = -3$. Mapping from $X_{22} \to X_{7} \to X_{1}$ gives
$j = -320$. The smallest conductor elliptic curve with this $j$-invariant is
\[
  E \colon y^{2} = x^{3} - x^{2} - 3x + 7,
\]
and the $2$-adic image for this curve is $H_{57}$.
\end{example}


\vspace{3pt}
\begin{figure}
  \centering
\begin{tabular}{|l|l|}

\hline
Type & Number\\
\hline
$X_{H} \cong \P^{1}$ & $175$\\
Pointless conics & $10$\\
Elliptic curves with positive rank & $27$\\
Elliptic curves with rank zero & $25$\\
Genus $1$ curves computed with no points & $6$\\
Genus $1$ curves whose models are not necessary & $165$\\
\hline
Genus $2$ models computed & $57$\\
Genus $2$ curves whose models are not necessary & $40$\\
Genus $3$ models computed & $22$\\
Genus $3$ curves whose models are not necessary & $142$\\
Genus $5$ models computed & $20$\\
Genus $5$ curves whose models are not necessary & $24$\\
Genus $7$ models computed & $4$\\
Genus $7$ curves whose models are not necessary & $10$\\
\hline
\end{tabular}
  \caption{Summary of the computation of the 727 models. }
\label{tab:curves}
\end{figure}


\section{The cases with $-I \not\in H$}
\label{sec:universal-curve}

In this section we describe how to compute, for
subgroups such that $-I \not\in H$ and $g(X_H) = 0$, a family of curves $E_t$ over an open subset $U \subset
\P^1$ such that an elliptic curve $E/K$ without CM has 2-adic image of
Galois contained in a subgroup conjugate to $H$ if and only if there
exists $t \in U(K)$ such that $E_t \cong E$. \\

When $-I \in H$, the $2$-adic image
for $E$ is contained in $H$ if and only if the same is true of the
quadratic twists $E_{D}$ of $E$. For this reason, knowing equations for the covering
map $X_{H} \to X_{1}$ is sufficient to check whether a given elliptic curve has $2$-adic image contained in $H$.

When $-I \not\in H$,
more information is required. First, observe that if $-I \not\in H$,
then $\widetilde{H} = \langle -I, H \rangle$ is a subgroup with
$[\widetilde{H} : H] = 2$ that contains $H$. Recall that the coarse spaces of  $X_{H}$ and $X_{\widetilde{H}}$ are isomorphic.
In order for there to be
non-trivial rational points on $X_{H}$, it must be the case that
$X_{\widetilde{H}}(\Q)$ contains non-cuspidal, non-CM rational points. A
detailed inspection of the rational points in the cases that $-I \in
H$ shows that this only occurs if $X_{\widetilde{H}}$ has genus zero. There
are $1006$ subgroups $H$ that must be considered.

Since we are not interested in the cases of elliptic curves
with CM, we will remove the points of $X_{H}$ lying over $j = 0$
and $j = 1728$. Let $\pi \colon X_{H} \to \P^{1}$ be the map to the $j$-line
and $U = \pi^{-1}(\P^1 - \{0,12^3, \infty\}) \subset X_H$. Then points of $U$  have no non-trivial
automorphisms and as a consequence, $U$ is fine moduli space (see Section \ref{sec:modular-curves}).
We let $E_{H} \to U$ denote the universal family of (non-CM) elliptic curves with
$2$-adic image contained in $H$. By Lemma \ref{L:universal-weierstrass-embedding} there is a model for $E_{H}$ of the form
\[
  E_{H} \colon y^{2} = x^{3} + A(t) x + B(t)
\]
where $A(t), B(t) \in \Z[t]$. 
Knowing that such a model exists, we will now describe how to find it.
\vspace{6pt}

Let $K$ be any field of characteristic zero. Suppose that $E/K$ is an
elliptic curve corresponding to a rational point on $X_H$  with $j(E)
\not\in \{0, 1728\}$ and given by
\[
  E \colon y^{2} = x^{3} + Ax + B.
\]
Now, if
\[
  E_{d} \colon dy^{2} = x^{3} + Ax + B
\]
is a quadratic twist of $E$, then $E$ and $E_{d}$ are isomorphic over
$K(\sqrt{d})$ with the isomorphism sending $(x,y) \mapsto (x,y/\sqrt{d})$.
Fix a basis for the $2$-power torsion points on $E$ and let
 $\rho_{E} \colon \Gal(\overline{K}/K) \to \GL_{2}(\Z_{2})$ be the corresponding Galois representation.
Taking the image of the fixed basis on $E$ under this isomorphism gives
a basis on $E_{d}$, and with this choice of basis, we have
\[
  \rho_{E_{d}} = \rho_{E} \cdot \chi_{d}
\]
where $\chi_{d}$ is the natural isomorphism
$\Gal(K(\sqrt{d})/K) \to \{ \pm I \}$. We can now state our next result.

\begin{lemma}
\label{L:uniqueness-of-twisting}
Assume the notation above. Let $\widetilde{H}$ be the subgroup generated by
the image of $\rho$ and $-I$.  Suppose $H \subset \widetilde{H}$ is a subgroup
of index $2$ with $-I \not\in H$. Then there is a unique quadratic twist
$E_{d}$ so that the image of $\rho_{E_{d}}$ (computed with respect to
the fixed basis coming from $E$) lies in $H$.
\end{lemma}
\begin{remark}
  Without the chosen basis for the $2$-power torsion on $E_{d}$, the
  statement is false. Indeed, it is possible for two
  different index two (and hence normal) subgroups $N_{1}$ and $N_{2}$
  of $\widetilde{H}$ to be conjugate in $\GL_{2}(\Z_{2})$. The choice
  of a different basis for the $2$-power torsion on $E_{d}$ would
  allow the image of $\rho_{E_{d}}$ to be either $N_{1}$ or $N_{2}$.
\end{remark}
\begin{proof}
Observe that $j(E) \not\in \{0, 1728\}$ implies that $E \cong E_{d}$ if and
only if $d \in (K^{\times})^{2}$. Recall that
$\rho_{E_{d}} = \rho_{E} \cdot \chi_{d}$.

Let $L$ be the fixed field of $\{ \sigma \in \Gal(\overline{K}/K) : \rho_{E}(\sigma) \in H \}$. Then since $H$ is a subgroup of $\widetilde{H}$ of index at most
$2$, $[L : K] \leq 2$. If $\rho_{E}(\sigma) \not\in H$,
then $\rho_{E}(\sigma) \in (-I) H$. Thus, the image of $\rho_{E_{d}}$ is
contained in $H$ if and only if $\chi_{d}(\sigma) = -1 \iff \sigma
\not\in \Gal(\overline{K}/L)$. Thus, the image of $\rho_{E_{d}}$ is contained
in $H$ if and only if $L = K(\sqrt{d})$. This proves the claim.
\end{proof}

We start by constructing a model for an elliptic curve
\[
  E_{t} \colon y^{2} = x^{3} + A(t) x + B(t)
\]
where $A(t), B(t) \in \Z[t]$ and $j(E_{t}) = p(t)$, where $p \colon X_{\widetilde{H}}
\to X_{1}$ is the covering map from $X_{\widetilde{H}}$ to the $j$-line. By the
above lemma, the desired model of $E_{H}$ will be a quadratic twist of
$E_{t}$, so
\[
  E_{H} \colon y^{2} = x^{3} + A(t) f(t)^{2} + B(t) f(t)^{3}
\]
for some squarefree polynomial $F(t) \in \Z[t]$. (Here, we say that a polynomial $F(t) \in \Z[t]$ is squarefree if whenever $F(t) = g(t)^2h(t)$ with $g,h \in \Z[t]$, then $g = \pm 1$.)

\begin{lemma}
\label{L:square-free-division-check}
  Let $F(t) \in \Z[t]$ be squarefree and let $D(t) \in \Z[t]$. Suppose that for all but finitely many $n \in \Z$, the squarefree part of $f(n)$ divides $D(n)$. Then $F(t)$ divides $D(t)$ in $\Q[t]$.
\end{lemma}

\begin{proof}



By the Chinese Remainder Theorem, the squarefree part of $F(n)$ takes infinitely many distinct values.
%
%
%
%
We proceed by induction on $\deg D(t)$. If $D(t)$ is constant then the statement is trivial, since we can choose $n$ such that the squarefree part of $F(n)$ has absolute value larger than $|D(n)|$, giving a contradiction unless $F(t)$ is also constant. In general, by the division algorithm we can write
\[
MD(t) = q(t)F(t) + r(t)
\]
for some $M \in \Z$ and $q(t),r(t) \in \Z[t]$ such that $\deg r(t) < \deg F(t)$. But then there are infinitely many $n$ such that the squarefree part of $F(n)$ divides $r(n)$; by induction, this is a contradiction unless $r(t)$ is identically zero, in which case we have
\[
MD(t) = q(t)F(t),
\]
completing the proof.

\end{proof}

\begin{theorem}
\label{discthmIntegral}
Let $F(t) \in \Z[t]$ be squarefree and such that $E_H$ is isomorphic to the twist $E_{t,F(t)}$ of $E_{t}$ by $F(t)$  and let $D(t)$ be the discriminant of the model $E_{t}$ given above.
Then $F(t) | D(t)$ in $\Q[t]$.
\end{theorem}

\begin{proof}
We specialize, picking $n \in \Z$ so that $E_{n}$ is non-singular.
The $2$-adic image for $E_{n}$ is contained in $\widetilde{H}$.
If $K$ is the fixed field of $H$, then $K/\Q$ is a trivial or quadratic extension.
If $\chi$ is the Kronecker character of $K$ (resp.~trivial character), then twisting $E_{n}$ by
$\chi$ will give a curve whose $2$-adic image is contained in $H$.

Since $K \subseteq \Q(E_{n}[2^{k}])$ for some $k$, $K$ must be unramified
away from $2$ and the primes dividing the conductor of $E_{n}$. Since
the conductor of $E_{n}$ divides the minimal discriminant of $E_{n}$,
and this in turn divides the discriminant of $E_{n} \colon y^{2} = x^{3} + A(n) x + B(n)$ (which is a multiple of $16$), we have that if $K = \Q(\sqrt{d})$
with $d$ squarefree, then $d | D(n)$. Moreover, $d$ must be the squarefree part of $F(n)$. The theorem now follows from Lemma \ref{L:square-free-division-check}.
\end{proof}

Here is a summary of the algorithm we apply to compute the polynomial $F(t)$. Throughout, we will write  $F(t) = cd(t)$, where $d(t)$ divides $D(t)$ in $\Z[t]$, $c \in \Q$ is squarefree, and $d(t)$ is not the zero polynomial mod any prime $p$.

\begin{enumerate}
\item We pick an integral model for $E_{t}$ and repeatedly choose integer
values for $t$ for which $E_{t}$ is non-singular and does not have complex
multiplication.
\item\label{item:1} For each such $t$, we compute a family of resolvent polynomials,
one for each conjugacy class of $\tilde{H}$, that will allow us to determine the conjugacy class
of $\rho_{E_{t},2^{k}}({\rm Frob}_{p})$. (See Appendix \ref{App:computing-Image} for a procedure to do this.)
\item We make a list of the quadratic characters corresponding to $\Q(\sqrt{d})$
for each squarefree divisor $d$ of $2 N(E_{t})$. All twists of $E_{t}$
with $2$-adic image contained in $H$ must be from this set.
\item We compute the $\GL_{2}(\Z_{2})$-conjugates of $H$ inside $\widetilde{H}$.
(For the $\widetilde{H}$ that we consider, computation reveals that there can be $1$, $2$, or $4$ of these.)
\item We use the resolvent polynomials to compute the image of ${\rm Frob}_{p}$
for several primes $p$. Once enough primes have been used, it is possible
to identify which twist of $E_{t}$ has its $2$-adic image contained
in each $\GL_{2}(\Z_{2})$-conjugate of $H$.
\item The desired model of $E_{t}$ will be a twist by $c d(t)$
for some divisor $d(t)$ of the discriminant. We keep a list of candidate
values for $c$ for each divisor $d(t)$ that work for all of the $t$-values
tested so far, and eliminate choices of $d(t)$.
\item We go back to the first step and repeat until the number of options
remaining for pairs $(c,d(t))$ is equal to the number of $\GL_{2}(\Z_{2})$-conjugates of $H$ in $\widetilde{H}$. Each of these pairs $(c,d(t))$ gives a model
for $E_{H}$. We output the simplest model found.
\end{enumerate}

\begin{remark}
The algorithm above (step \ref{item:1} in particular) sometimes requires a lot of decimal precision (in some cases as much as $8500$ digits), and is in general fairly slow. Computing the equation for the universal curve over $X_H$ is thus  much slower
than computing equations for $X_H$ when  $-I \in H$.
\end{remark}

\begin{example}
There are two index $2$ subgroups of $H_{57}$ that do not contain $-I$. One of
these, which we call $H_{57a}$, contains $\Gamma(32)$, and is generated by
\[
  \left[ \begin{matrix} 10 & 21 \\ 3 & 13 \end{matrix} \right],
  \left[ \begin{matrix} 15 & 1 \\ 27 & 2 \end{matrix} \right],
  \left[ \begin{matrix} 7 & 7 \\ 0 & 1 \end{matrix} \right].
\]
We will compute $E_{H}$, the universal elliptic curve over $H_{57a}$. We let
\[
  E_{t} \colon y^{2} = x^{3} + A(t)x + B(t),
\]
where
\begin{align*}
  A(t) &= -6(725t^{8} + 1544t^{7} + 2324t^{6} + 2792t^{5} + 2286t^{4} + 1336t^{3} + 500t^{2} + 88t + 5)\\
  B(t) &= -32(3451t^{12} + 11022t^{11} + 22476t^{10} + 35462t^{9} + 43239t^{8} + 41484t^{7} + 32256t^{6} + 19596t^{5}\\
  &+ 8601t^{4} + 2630t^{3} + 564t^{2} + 78t + 5).
\end{align*}
These polynomials were chosen so that
\[
  j(E_{t}) = \frac{2^{6} (25t^{4} + 36t^{3} + 26t^{2} + 12t + 1)^{3} (29t^{4} + 20t^{3} + 34t^{2} + 28t + 5)}{(t^{2} - 2t - 1)^{8}} = p(t),
\]
where $p \colon X_{57} \to X_{1}$ is the map to the $j$-line. There are four
squarefree factors of the discriminant of $E_{t}$ in $\Q[t]$:
\[
  1, t^{2} - 2t - 1, t^{4} + \frac{20}{29} t^{3} + \frac{34}{29} t^{2} + \frac{28}{29} t + \frac{5}{29}, \text{ and }
  t^{6} - \frac{38}{29} t^{5} - \frac{35}{29} t^{4} - \frac{60}{29} t^{3} - \frac{85}{29} t^{2} - \frac{38}{29} t - \frac{5}{29}.
\]
We specialize $E_{t}$ by taking $t = 1$, giving
\[
  E_{t} \colon y^{2} = x^{3} - 69600 x + 7067648.
\]
Considering $H_{57}$ as a subgroup of $\GL_{2}(\Z/32\Z)$, it has $416$ conjugacy
classes. We compute the resolvent polynomials for each of these conjugacy
classes and verify that they have no common factors. Since $E_{t}$ has
conductor $2^{8} \cdot 3^{2} \cdot 29^{2}$, the fixed field of $H_{57a}$
inside $\Q(E_{t}[32])$ is a quadratic extension ramified only at $2$, $3$
and $29$. There are sixteen such fields.

There are two index $2$ subgroups of $H_{57}$ that are
$\GL_{2}(\Z_{2})$-conjugate to $H_{57a}$. As a consequence, there are
two quadratic twists of $E_{t}$ whose $2$-adic image will be contained
in some conjugate of $H_{57a}$. By computing the conjugacy class of
$\rho({\rm Frob}_{p})$ for $p = 53$, $157$, $179$ and $193$, we are
able to determine that those are the $-87$ twist and the $174$ twist.
This gives us a total of $8$ possibilities for pairs $(c,d(t))$ (two
for each $d(t)$).

Next, we test $t = 2$. This gives the curve
\[
  E_{t} \colon y^{2} = x^{3} - 4024542x + 3107583520.
\]
This time, we find that the $-4926$ and $2463$ twists are the ones
whose $2$-adic image is contained in $H_{57a}$ (up to conjugacy). This
rules out all the possibilities for the pairs $(c,d(t))$ except for
two. These are $c = 174$ and $c = -87$ and $d(t) = t^{6} -
\frac{38}{29} t^{5} - \frac{35}{29} t^{4} - \frac{60}{29} t^{3} -
\frac{85}{29} t^{2} - \frac{38}{29} t - \frac{5}{29}$. This gives the
model
\[
  E_{H_{57a}} \colon y^{2} = x^{3} + \tilde{A}(t) x + \tilde{B}(t),
\]
where
\begin{align*}
  \tilde{A}(t) &= 2 \cdot 3^{3} \cdot (t^{2} - 2t - 1)^{2} \cdot (25t^{4} + 36t^{3} + 26t^{2} + 12t + 1) (29t^{4} + 20t^{3} + 34t^{2} + 28t + 5)^{3}\\
  \tilde{B}(t) &= 2^{5} \cdot 3^{3} \cdot (t^{2} - 2t - 1)^{3} \cdot (t^{2} + 1) \cdot (7t^{2} + 6t + 1) \cdot (17t^{4} + 28t^{3} + 18t^{2} + 4t + 1)\\ \cdot
 & (29t^{4} + 20t^{3} + 34t^{2} + 28t + 5)^{4}.
\end{align*}
In total, this calculation takes 3 hours and 46 minutes.

The smallest conductor that occurs in this family is $6400$. The
curve $E \colon y^{2} = x^{3} + x^{2} - 83x + 713$ and its $-2$-quadratic twist
$E' \colon y^{2} = x^{3} + x^{2} - 333x - 6037$ both have conductor $6400$ and
$2$-adic image $H_{57a}$.
\end{example}

\section{A curious example}
\label{sec:curious-example}

Before our exhaustive analysis of the rational points on the various $X_H$, we pause to discuss the following curious example, which demonstrates that Hilbert's irreducibility theorem does not necessarily hold when the base is an elliptic curve with positive rank.\\

One expects that if $X_H(\Q)$ is infinite then there exist infinitely many elliptic curves $E/\Q$ such that $\rho_E(G_{\Q})$ is actually \emph{equal} to $H$. The following example shows that this is not necessarily true.

\begin{example}
\label{ex:curious-example}
The subgroup $H_{155}$ is an index $24$ subgroup containing $\Gamma(16)$ generated by
\[
  \left[ \begin{matrix} 1 & 3 \\ 0 & 3 \end{matrix} \right],
  \left[ \begin{matrix} 1 & 0 \\ 2 & 3 \end{matrix} \right],
  \left[ \begin{matrix} 1 & 3 \\ 12 & 3 \end{matrix} \right],
  \text{ and } \left[ \begin{matrix} 1 & 1 \\ 12 & 7 \end{matrix} \right].
\]
The curve $X_{155}$ is an elliptic curve
\[
  X_{155} \colon y^{2} = x^{3} - 2x
\]
and $X_{155}(\Q) \cong \Z/2\Z \times \Z$ and is generated by
$(0,0)$ and $(-1,-1)$. The map from $X_{155}$ to the $j$-line is given
by $j(x,y) = \frac{256 (x^{4} - 1)^{3}}{x^{4}}$.

Since the two-torsion subgroup $X_{155}(\Q)[2]$ is
non-trivial (as imposed by Remark \ref{R:remark-torsion}), $X_{155}$
has an \'etale double cover $\phi\colon E \to X_{155}$ defined over
$\Q$ and such that $E$ has good reduction away from 2. By the
Riemann-Hurwitz formula, $E$ has genus 1; the map is thus a 2-isogeny
and $E(\Q)$ thus has rank one. By \'etale descent (see Subsection
\ref{ssec:etale-descent}), since $X_{155}$ and $E$ have good reduction outside of
2, every point of $X_{155}(\Q)$ lifts to $E_d(\Q)$ for $d \in \{\pm 1, \pm
2\}$. It turns out that for each such $d$, there is an index 2
subgroup $H_d \subset H_{155}$ such that $E_d \cong X_{H_d}$. (These are
$X_{284}$, $X_{318}$, $X_{328}$ and $X_{350}$, respectively.) It follows that
\[
\bigcup_{d \in  \{\pm 1, \pm 2\}} \phi_d(E_{d}(\Q))   = X_{H_{155}}(\Q).
\]
In particular, for every point in $X_{H_{155}}(\Q)$, the 2-adic image of Galois of the corresponding elliptic curve is contained in one of the four index two subgroups $H_d$!
\end{example}

\begin{remark}
  We note that if $X_H \cong \P^1$, then since $\P^1$ has no \'etale
  covers and since there is a finite collection of subgroups
  $H_1,\ldots,H_n$ such that any $K$ properly contained in $H$ is a subgroup
  of some $H_i$, the image
\[
\bigcup_{K \subset H} \phi_{K}(X_{K}(\Q)) = \bigcup_{i = 1}^n \phi_{H_i}(X_{H_i}(\Q))   \subset X_H(\Q)
\]
(where $\phi_{K}$ is the map $X_K \to X_H$ induced by the inclusion $K \subset H$) is a thin set, and in particular most (i.e.~a density one set) of the points of $X_H(\Q)$
correspond to $E/\Q$ such that $\rho_E(G_{\Q}) = H$.
\end{remark}


\begin{remark} There are seven genus one curves $X_{H}$ that are elliptic curves of positive rank where
the corresponding subgroup $H$ has index $24$.
In all seven cases, all
of the rational points lift to modular double covers (although it is
not always the case that all four twists have local points). In fact,
every one of the $20$ modular curves $X_{H}$, where $H$ has index $48$
and for which $X_{H}(\Q)$ is a positive rank elliptic curve is a double cover of one of these seven curves.
\end{remark}

This example is more than just a curiosity; it inspired the technique
of Subsection \ref{ssec:etale-descent-via} which allows us to
determine the rational points on most of the genus 5 and 7 curves.\\

This example also raises the following question.

\begin{question}
  Do there exist infinite unramified towers of modular curves such that each twist necessary for \'etale descent is modular?
\end{question}

If so, this would imply that none of the curves in such a tower have non-cuspidal non-CM points.  A potential example is the following: the Cummins/Pauli database \cite{MR2016709} reveals that there might be such a tower starting with $16A^2, 16B^3, 16B^5, 16B^9, 16A^{17}$. There is then a level 32, index 2 subgroup of $16A^{17}$ that has genus 33.

\section{Analysis of Rational Points - theory}

The curves whose models we computed above have genera either
0,1,2,3,5,7; see Table \ref{tab:curves}.\\


For the genus 0 curves, we determine whether the curve has a rational point, and if so we compute an explicit isomorphism with $\P^1$. For the genus 1 curves, we determine whether the curve has a rational point, and if so compute a model for the resulting elliptic curve and determine its rank and torsion subgroup. This is straightforward: all  covering maps except 4 have degree 2, so we end up with a model of the form $y^2 = p(t)$, where $p(t)$ is a polynomial, and the desired technique is implemented in Magma. The remaining 4 cases are handled via a brute force search for points.

In the higher genus cases, we determine the complete set of rational points. Each of the following techniques play a role:

\begin{enumerate}
\item local methods,
\item Chabauty for genus 2 curves,
\item elliptic curve Chabauty,
\item \'etale descent,
\item ``modular'' \'etale double covers of genus 5 and 7 curves, and
\item an improved algorithm for computing automorphisms of curves.
\end{enumerate}

In this section we describe in detail the theory behind the techniques
used to analyze the rational points on the higher genus  curves. The remainder of the paper is a case by case analysis of the rational points on the various $X_H$.

\begin{remark}[Facts about rational points on $X_H$]

  \begin{enumerate}
  \item []
  \item Every rational point on a curve $X_{H}$ of genus one that has rank zero is a cusp or a CM point.
  \item The only genus $2$ curve with non-cuspidal, non-CM rational points is $X_{441}$, also known as $X_{ns}^{+}(16)$. This curve has two non-cuspidal, non-CM rational points, with distinct $j$-invariants.

  \item The only genus $3$ curves with non-cuspidal, non-CM rational points are $X_{556}$, $X_{558}$, $X_{563}$, $X_{566}$, $X_{619}$, $X_{649}$. Each of these gives rise to a single, distinct $j$-invariant.

  \item All the rational points on the genus $5$ and $7$ curves are either cusps or CM points.


  \end{enumerate}

\end{remark}

\begin{remark}
\label{R:remark-torsion}
The following observation powers many of these approaches -- since
Jacobians of 2-power level modular curves have good reduction outside
of 2, each Jacobian is ``forced'' to have a non-trivial two torsion
point (and more generally forced to have small mod 2 image of
Galois). Indeed, the two division field $\Q(J[2])$ is unramified
outside of 2, and there are few such extensions of small degree. In
\cite{MR2643896}, it is shown that if $[K : \Q] \leq 16$ and $K/\Q$ is
ramified only at $2$, then $[K : \Q]$ is a power of $2$. In
particular, there are no degree 3 or 6 extensions of $\Q$ ramified
only at 2, so an elliptic curve with conductor a power of 2 has a
rational 2-torsion point.  (In practice of course one can often
compute directly the torsion subgroup of the Jacobian, by computing
the torsion mod several primes, and then explicitly finding
generators.) We remark that there is, however, a degree $17$ extension
of $\Q$ ramified only at $2$, arising from the fact that the class
number of $\Q(\zeta_{64})$ is $17$.
\end{remark}

\subsection{Chabauty}

See \cite{McCallumP:chabautySurvey} for a survey. The practical output is that
if $\rk \Jac_X(\Q) < \dim \Jac_X = g(X)$, then $p$-adic integration produces explicit 1-variable power series $f \in \Q_p\llbracket t \rrbracket$ whose set of  $\Z_p$-solutions contains all of the rational points. This is all implemented in Magma for genus 2 curves over number fields, which will turn out to be the only case needed. See the section below on genus 2 curves for a complete discussion.

\subsection{Elliptic Chabauty}
\label{ssec:elliptic-chabauty}

Given an elliptic curve $E$ over a number field $K$ of degree $d > 1$ over $\Q$ and a map $E \xrightarrow{\pi} \P^1_K$,  one would like to determine the subset of $E(K)$ mapping to $\P^1(\Q)$ under $\pi$. A method analogous to Chabauty's method provides a partial solution to this problem under the additional hypothesis that $\rank E(K) < d$ (and has been completely implemented in Magma).
The idea is to expand the map $E \to \P^1_K$ in $p$-adic power series and analyze the resulting system of equations using Newton polygons or similar tools.
See \cite{Bruin:ellChab}, \cite{bruin:nn2} for a succinct description of the method and instructions for use of its Magma implementation.

A typical setup for applications is the following.
\[
\xymatrix{
C\ar[rd]\ar[dd]_{\phi} &  \\
& E \ar[dl]^{\psi}\\
\P^1& \\
}
\]
We have a higher genus curve $C$ whose rational points we want to determine, and we have a particular map $C \to \P^1$ which is defined over $\Q$ and which factors through an elliptic curve $E$ over a number field $K$ (but does not necessarily factor over $\Q$). Then any $K$-point of $E$ which is the image of a $\Q$-point of $C$ has rational image under $E \to \P^1$, exactly the setup of elliptic curve Chabauty. (Finding the factorization $C \to E$ can be quite tricky; see Subsection \ref{ssec:analysis-x_619} for an example.)

\subsection{\'Etale descent}
\label{ssec:etale-descent}

\'Etale descent is a ``going up'' style technique, first studied in \cite{coombesG:heterogeneous} and \cite{wetherell1:thesis} and developed as a full theory (especially the non-abelian case) in \cite{Skorobogatov:torsorsBook}. It is now a standard technique for resolving the rational points on curves (see e.g.~\cite{Bruin:ellChab}, \cite{flynnW:challengeProblem}) and lies at the heart of the modular approach to Fermat's last theorem (see \cite[5.6]{Poonen:bookNumberTheory}).

Let $\pi\colon X \to Y$ be an \'etale cover defined over a number field $K$ such that $Y$ is the quotient of some free action of a group $G$ on $X$. Then there exists a finite collection $\pi_1\colon X_1 \to Y,\ldots,\pi_n\colon X_n \to Y$ of twists of $X\to Y$ such that
\[
\bigcup_{i = 1}^n \pi_i(X_{i}(K))   = Y(K).
\]
Moreover, if we let $S$ be the union of the set of primes of bad reduction of  $X$ and $Y$ and of the primes of $\calO_K$ over the primes dividing $\#G$, then the cocycles corresponding to the twists are unramified outside of $S$. (See e.g.~\cite[5.3]{Skorobogatov:torsorsBook}.)

We will use this procedure only in the case of \'etale double covers. In this case, $G = \Z/2\Z$ and, since the twists are consequently quadratic, we will instead denote twists of a double cover $X \to Y$ by $X_d \to Y$, where $d \in K^{\times}/\left(K^{\times}\right)^{2}$, and the above discussion gives that, for any point $P$ of $Y(K)$, there will exist $d \in \calO_{K,S}^{\times}/(\calO_{K,S}^{\times})^2$ such that $P$ lifts to a point of $X_d(K)$.

\subsection{\'Etale descent via double covers with modular twists}
\label{ssec:etale-descent-via}

The following variant of Example \ref{ex:curious-example}  will allow us to resolve the rational points on some of the high genus curves.\\

We will occasionally be in the following setup: $K \subset H \subset \GL_2(\Z_2)$ are a pair of open subgroups such that $g(X_H) > 1$ and the corresponding map $X_{K} \to X_H$ is an \'etale double cover.
By \'etale descent (see Subsection \ref{ssec:etale-descent}), since
$X_H$ and $X_{K}$ have good reduction outside of 2, every point of
$X_H(\Q)$ lifts to a rational point on a quadratic twist $X_{K,d}(\Q)$ for $d \in \{\pm 1, \pm 2\}$, so that
\[
\bigcup_{d \in  \{\pm 1, \pm 2\}} \phi_d(X_{K,d}(\Q))   = X_H(\Q),
\]
where $X_{K} \to X_H$ is induced by the inclusion $K \subset H$ and $\phi_{K,d}$ is the twist of this by $d$.


It turns out that, additionally, for each such $d$ there is an index 2 subgroup $K_d \subset H$ such that $X_{K_d} \cong X_{K,d}$; i.e.~each of the quadratic twists are also modular. Finally, a third accident occurs: each of the subgroups $K_d$ is contained in a subgroup $L_d$ such that $X_{L_d}$  either has genus 1 and has no rational point, is an elliptic curve of rank zero, or is a genus zero curve with no rational points. In particular, since the inclusion of subgroups $K_d \subset L_d$ induces a map $X_{K_d} \to X_{L_d}$, this determines all of the rational points on each twist $X_{K_d}$, and thus on $X_{H}$.
\\

This phenomenon occurs for $16$ of the $20$ subgroups $H$ for which $X_{H}$ has
genus $5$, and all four of the cases when $X_{H}$ has genus $7$. See
Subsection \ref{ssec:non-explicit-modular} for details.

 \subsection{Constructing automorphisms of curves over number fields}
\label{ssec:automorphismGroups}

If $C$ is a curve of genus $g$ and $D \to C$ is a degree $n$ \'etale cover
of $C$, then the genus of $D$ is $ng - (n-1)$. In order to analyze rational
points on $D$, it is very helpful to be able to find maps from $D$ to
curves of lower genus. In this context, it is helpful to compute
the group $G$ of automorphisms of $D$ and consider quotients $D/H$
for subgroups $H \subseteq G$.

Magma's algebraic function field machinery is able to compute automorphism
groups of curves. However, the performance of these routines varies quite
significantly based on the complexity of the base field. The routines
work quickly over finite fields, but are often quite slow over number
fields, especially when working with curves that have complicated models.

For our purposes, we are interested in quickly constructing automorphisms
(defined over $\Qbar$) of non-hyperelliptic curves $D/\Q$  with genus $\geq 3$.
(Magma has efficient, specialized routines for genus 2 and genus 3 hyperelliptic curves.) Our goal is not to provably compute the automorphism group, but
to efficiently construct all the automorphisms that likely exist. The procedure
we use is the following.

\begin{enumerate}
\item Given a curve $D/\Q$, use Magma's routines to compute $\Aut(D/\F_{p})$
for several different choices of primes $p$. If all automorphisms of $D$ are defined over the number field $K$,
then we expect that if $p$ splits completely in $K$, then $|\Aut(D/\F_{p})| = |\Aut_{\Qbar}(D)|$. Data for several
primes will give a prediction for $|\Aut_{\Qbar}(D)|$ and $K$.
\item Consider the canonical embedding of $D \subset \P^{g-1}$. Any automorphism
of $C$ can be realized as a linear automorphism of $\P^{g-1}$ that fixes the canonical image of $D$.
\item Construct the ``automorphism scheme'' $X/\Q$ of linear automorphisms
from $\P^{g-1}$ that map $D$ to itself. Let $I(D) \subseteq \Q[x_{1},x_{2},\ldots,x_{g}]$
denote the ideal of polynomials that vanish on the canonical image of $C$. For each
homogeneous generator $f_{i}$ of $I(D)$ of degree $d_{i}$, we construct
a basis $v_{1}^{(i)}, v_{2}^{(i)}, \ldots, v_{e_{i}}^{(i)}$ for the degree $d_{i}$ graded piece of $I(D)$. If $\phi \colon D \to D$ is an automorphism, then
\[
  \phi(f_{i}) = \sum_{j=1}^{e_{i}} c_{i,j} v_{j}^{(i)}.
\]
We construct the automorphism scheme as a subscheme of $\A^{d}$, where
$d = g^{2} + \sum_{i} d_{i} + 1$. We use $g^{2}$ variables for the linear transformation, $\sum_{i} d_{i}$ variables for the constants $c_{i,j}$ in the above equation, and one further variable to encode the multiplicative inverse of the determinant of the linear transformation. (This scheme actually has dimension $1$ since an arbitrary scaling of
the matrix is allowed.) We will extend $X$
to a scheme over $\Spec \Z$ (which we also call $X$).
\item Choose a prime $p$ that splits completely in $K$ and a prime
  ideal $\mathfrak{p}$ of norm $p$ in $\mathcal{O}_{K}$, the ring of
  integers in $K$.  Use Magma's routines to compute $\Aut(D/\F_{p})$
  and represent these automorphisms as points in $X(\F_{p})$.
\item Use Hensel's lemma to lift the points on $X(\F_{p})$ to points on
$X(\Z/p^{r} \Z)$ for some modestly sized integer $r$. (We frequently use $r = 60$.) Hensel's lemma is already implemented in Magma via {\tt LiftPoint}.
\item Scale the lifted points so that one nonzero coordinate is equal to $1$.
Then use lattice reduction to find points in $K$ of small height
that reduce to the points in $X(\Z/p^{r} \Z)$ modulo $\mathfrak{p}^{r}$.
Use these to construct points in $X(K)$, i.e., automorphisms of $D$ defined
over $K$.
\end{enumerate}

The above algorithm runs very quickly in practice for curves of reasonably small genus. For example, the genus $5$ curve given by
\begin{align*}
  -2705a^{2} + 1681b^{2} - 1967bc + 2048c^{2} - 2d^{2} &= 0\\
  73a^{2} - 41b^{2} + 64bc - 64c^{2} - 2de &= 0\\
  -2a^{2} + b^{2} - 2bc + 2c^{2} - 2e^{2} &= 0
\end{align*}
is one of the \'etale double covers of $X_{619}$. This curve has (at least) $16$ automorphisms defined over $\Q(\sqrt{2 + \sqrt{2}})$ which are found by the above algorithm in $25.6$ seconds. However, Magma's built in routines require a
long time to determine the automorphism group (the routine did not finish
after running it for 3 and 1/2 days).

 \subsection{Fast computation of checking isomorphism of curves}
\label{ssec:testingIsomorphism}

A related problem to computing automorphisms is proving that two curves are isomorphic. There are many instances of non-conjugate subgroups $H$ and $K$ with $X_{H} \cong X_{K}$. Within the $22$ genus three curves, there are at most $7$ isomorphism classes. Within the $20$ genus five curves, there are at most $10$ isomorphism classes. The $4$ genus seven curves fall into two isomorphism classes.

Magma's built-in command {\tt IsIsomorphic} suffices for hyperelliptic curves and a few higher genus curves that happen to have nice models. The simplest
way to determine if two non-hyperelliptic genus $3$ curves are isomorphic is to
compute their canonical models and apply {\tt MinimizeReducePlaneQuartic} and inspect
the resulting simplified polynomials - at this point the isomorphisms can be seen by inspection.

In the genus $5$ case, we use a variant of the approach described for automorphisms,
and, given two curves $C_{1}$ and $C_{2}$, we construct an ``isomorphism scheme'' in
a similar way to the automorphism scheme above. Again, we use Magma's internal commands to find isomorphisms mod $p$, and lift these to characteristic zero isomorphisms.
In the genus $7$ case, Magma's built-in commands are the most efficient.

 \subsection{Probable computation of ranks}
\label{ssec:probableComputationOfRanks}

It is straightforward to compute the rank of a curve of genus at most 2 using Magma's preexisting commands (e.g.~via \verb+RankBound+, an implementation of \cite{stoll:hyperellipticDescent}); computation of the rank of the Jacobian of a genus 3 plane curve has recently been worked out \cite{bruinPS:genus3Descent}, but is often impractical \cite{bruinPS:genus3Descent}*{Remark 1.1} and moreover has not been implemented in a publicly available way. For genus $> 3$ little is known in general (though special cases such as cyclic covers of $\P^1$ are known \cite{PoonenS:explicitDescentFor}, \cite{StollL:selmerCyclic}).

For the determination of the rational points on each $X_H$, we will only need a rigorous computation of rank for genus at most 2. Nonetheless, in many cases we can compute ``probable'' ranks, and mention this in the discussion as an indication of why we chose a particular direction of analysis. If $H$ is a subgroup of $\GL_{2}(\Z_{2})$
that contains $\Gamma(2^{k})$, then $X_{H}$ is a quotient of $X_{1}(4^{k})$, but the map from $X_{1}(4^{k}) \to X_{H}$ is only defined over $\Q(\zeta_{4^{k}})$. For
this reason, we cannot immediately conclude that each factor $A$ of $\Jac X_{H}$ is modular. However, numerical data suggests that each such $A$ is indeed a factor
of $\Jac X_{1}(4^{k})$. We can find a candidate for the corresponding modular form $f$ (e.g.~by comparing traces) and compute a guess for the analytic rank, but we cannot prove that $A \cong A_f$, or that the algebraic and analytic ranks of $A_f$ agree.

\section{Analysis of Rational Points - Genus 2}

In the remaining sections we provably compute all of the rational points on each modular curve. Magma code verifying the below claims is available at \cite{RouseZB:2adicTranscript} and additionally at the arXiv page of this paper.\\

There are 57 arithmetically maximal genus 2 curves. Among these, 46  have Jacobians with rank 0, 3 with rank 1, and 8 with rank 2. We will use \'etale descent on the rank 2 cases and Chabauty on the others. In each case, the rank of the Jacobian is computed with Magma's $\verb+RankBound+$ command. See the transcript of computations for full details, and see \cite{bruinS:wee} for a detailed discussion of all practical techniques for determining the rational points on a genus 2 curve.

\subsection{Rank 0}

If $\rk \Jac_X(\Q) =0$ then $\Jac_X(\Q)$ is  torsion. To find all of the rational points on $X$ it thus suffices to compute the torsion subgroup of $\Jac_X(\Q)$ and compute preimages of these under an inclusion $X \hookrightarrow \Jac_X$.  This is implemented in Magma as the $\verb+Chabauty0(J)+$ command, and
in each case Magma computes that the only rational points are the known points.

\subsection{Rank 1}
If $\rk \Jac_X(\Q) =1$ then one can attempt Chabauty's method.
This is implemented in Magma as the $\verb+Chabauty(ptJ)+$ command, and
in each case Magma computes that the only rational points are the known points.

\subsection{Rank 2}
\label{ssec:rank2}
If $\rk \Jac_X(\Q) =2$ then Chabauty's method doesn't apply and the analysis is more involved; instead we proceed by \'etale descent. In each case, the Jacobian of $X$ has a rational 2-torsion point. Thus, given a model
\[
X \colon y^2  = f(x)
\]
of $X$, $f$ factors as $f_1f_2$, where both are polynomials of positive  degree (and both of even degree if $f$ has even degree), and $X$ admits  \'etale double covers $C_d \to X$, where  the curve $C_d$ is given by
\[
\begin{array}{crl}
  C_d\colon &dy_1^2 \, =   & f_1(x) \\
  &dy_2^2 \, = & f_2(x)
\end{array}
\]
Since $X$ has good reduction outside of 2 and the 2-cover $C_1 \to X$ is \'etale away from 2 (since it is the pullback of a 2-isogeny  $A \to \Jac_X$, and such an isogeny is \'etale away from 2), by \'etale descent (see \ref{ssec:etale-descent} above) every rational point on $X$ lifts to a rational point on $C_d(\Q)$ for $d \in \{\pm 1, \pm 2\}$. The Jacobian of $C_d$ is isogenous to $\Jac_X \times E_d$, where $E_d$ is the Jacobian of the (possibly pointless) genus one curve  $dy_2^2 = f_2(x)$ (where we assume that $\deg f_2 \geq \deg f_1$, so that $\deg f_2 \geq 3$).

There are 4 isomorphism classes of genus 2 curves in our list with Jacobian of rank 2 ($X_{395},X_{402},X_{441},X_{520}$). In two cases ($X_{395}$ and $X_{402}$), each twist $C_d$ maps to a rank 0 elliptic curve. For example,
$X_{395}$ is the hyperelliptic curve $y^2 = x^6 - 5x^4 - 5x^2 + 1 = (x^2 - 2x - 1)(x^2 + 1)(x^2 + 2x - 1)$. This admits \'etale covers by the genus 3 curves
\[
\begin{array}{crl}
  C_d\colon &dy_1^2 \, =   & x^2 + 1\\
  &dy_2^2 \, = & (x^2 - 2x - 1)(x^2 + 2x - 1)
\end{array}
\]
each of which in turn maps to the genus 1 curve  $E_{d} : dy_2^2 = (x^2 - 2x - 1)(x^2 + 2x - 1)$, and for $d \in \{\pm 1, \pm 2\}$, $\rk \Jac_{E_d} = 0$, allowing the determination the rational points on each $C_d$ and thus on $X_{395}$.

For the remaining genus 2 curves, three of the twists map to a rank 0 elliptic curve, but the twist by $-2$ maps to a rank 1 elliptic curve.  Here one may apply \'etale descent again, but over a quadratic extension.
For example,  $X_{441}$ is the hyperelliptic curve $y^2 = x^6 - 3x^4 + x^2 + 1 =(x - 1)(x + 1)(x^4 - 2x^2 - 1)$. (This is the curve $X^+_{\text{ns}}(16)$ whose non-cuspidal points classify elliptic curves whose mod 16 image of Galois is contained in the normalizer of a non-split Cartan subgroup. The rational points on this curve are resolved in \cite{baran:normalizersClassNumber} via elliptic Chabauty; we give an independent determination of the rational points on this curve.)
This admits  \'etale covers by the genus 3 curves
\[
\begin{array}{crl}
  C_d\colon &dy_1^2 \, =   & (x - 1)(x + 1)\\
  &dy_2^2 \, = & (x^4 - 2x^2 - 1)
\end{array}
\]
The Jacobian of $dy_2^2 = x^4 - 2x^2 - 1$ has rank 0 for $d = \pm 1, 2$. For $d = -2$, we note that since $x^4 - 2x^2 - 1$ factors over $\Q(\sqrt{2})$ as $((x-1)^2-\sqrt{2}) ((x-1)^2+\sqrt{2})$, $C_{-2}$ admits a further \'etale double cover over $\Q(\sqrt{2})$ by
\[
\begin{array}{crl}
  X_{-2,d'}\colon &-2y_1^2 \, =   & (x - 1)(x + 1)\\
  &-2d'y_2^2 \, = & (x-1)^2-\sqrt{2}\\
  &d'y_3^2 \, = & (x-1)^2+\sqrt{2}
\end{array}
\]
(Note that a priori one expects this factorization to occur over a small field by Remark \ref{R:remark-torsion}.) By descent theory, every rational point on $C_{-2}$ lifts to a $K := \Q(\sqrt{2})$ point on  $X_{-2,d'}$ for some $d' \in \calO_{K,S}^{\times}/\left(\calO_{K,S}^{\times}\right)^{2}$.
These each map to the two genus 1 curves $d' y^2 = (x - 1)(x + 1)((x-1)^2 - \sqrt{2})$ and
$-2d' y^{2} = (x-1)(x+1) ((x-1)^{2} + \sqrt{2})$. 
For 6 of the 8 such $d'$, one of these curves has rank 0, and  for 2 both have rank 1. Any point coming from a rational point on $X_{441}$ has rational $x$-coordinate, and elliptic Chabauty (as described in Subsection \ref{ssec:elliptic-chabauty}) successfully resolves the rational points on the remaining two curves.


\section{Analysis of Rational Points - Genus 3}


There are 18 genus 3 curves (and at most 7 isomorphism classes).
Of the isomorphism classes, $X_{556}, X_{558}$ are hyperelliptic and handled by \'etale descent; $X_{618}$ admits a map to a rank zero elliptic curve defined over $\Q(\sqrt{2})$; $X_{628}$, $X_{641}$, and $X_{650}$ have nice models and can be handled in a direct, ad hoc manner. Finally, $X_{619}$ is the most difficult case --  it has six rational points and its Jacobian has (probable) analytic rank 3; we are nonetheless able to handle this curve via an elliptic Chabauty argument whose setup is non-trivial. All other genus 3 curves on our list are isomorphic to one of these.

\begin{remark}
Unfortunately, consideration of Prym varieties (see \cite{bruin:Prym} for a discussion) do not simplify analysis of any of the above curves; for instance, $X_{619}$ admits an \'etale double cover, but one of the twists of the associated Prym varieties has rank 2.
\end{remark}



\subsection{Genus 3 hyperelliptic}

The genus 3 curves $X_{556}, X_{558}, X_{563}, X_{566}$ are hyperelliptic. The last two curves are isomorphic to the first two, which are given by
\[
\begin{array}{cl}
  X_{556}\colon &y^2 = x^7 + 4x^6 - 7x^5 - 8x^4 + 7x^3 + 4x^2 - x \\
  X_{558}\colon & y^2 = x^8 - 4x^7 - 12x^6 + 28x^5 + 38x^4 - 28x^3 - 12x^2 + 4x + 1
\end{array}
\]
Their Jacobians have rank 1, but unfortunately much of the machinery  necessary to do Chabauty on curves of genus $g > 2$ is not implemented in Magma
(e.g., a simple search did not reveal generators for the Jacobian of $X_{556}$; for a genus 2 curve one can efficiently search on the associated Kummer surface, but the analogous computation for abelian threefolds is not implemented).

Instead, we proceed by descent. The hyperelliptic polynomials both factor, so each $X$ admits an \'etale double cover which itself admits a map to a genus 2 curve. Rational points on the genus 3 curves lift to twists of the \'etale double cover by  $d \in \{\pm 1, \pm2\}$.
For example, $X_{556}$ admits  \'etale double covers by the genus 5 curves
\[
\begin{array}{crl}
  C_d\colon &dy_1^2 \, =   & x\\
  &dy_2^2 \, = & (x - 1)(x + 1)(x^4 + 4x^3 - 6x^2 - 4x + 1)
\end{array}
\]
which each maps to the genus 2 hyperelliptic curve
\[
H_d\colon dy^2 = (x - 1)(x + 1)(x^4 + 4x^3 - 6x^2 - 4x + 1).
\]
For $d \in \{\pm 1, \pm 2\}$ the Jacobian of $H_d$ has rank 0 or 1, and Chabauty reveals that any rational point on $X_{556}$ is either a point at infinity
or satisfies $x = 0$ or $y = 0$. Similarly, the defining polynomial of $X_{558}$ factors as $(x^2 - 2x - 1)(x^2 + 2x - 1)(x^4 - 4x^3 - 6x^2 + 4x + 1)$, and each of the four resulting genus 2 hyperelliptic curves 
\[
  dy^{2} = (x^{2} + 2x - 1)(x^{4} - 4x^{3} - 6x^{2} + 4x + 1)
\]
have Jacobians of rank 1.

Each of these four hyperelliptic curves has four non-cuspidal, non-CM rational
points that all have the same image on the $j$-line. For $X_{556}$ we obtain
$j = 2^{4} \cdot 17^{3}$, for $X_{558}$ we obtain $j = \frac{4097^{3}}{16}$,
for $X_{563}$ we obtain $j = 2^{11}$, and for $X_{566}$ we obtain $j = \frac{257^{3}}{256}$.

\subsection{Analysis of $X_{618}$}
 The curve $X_{618}$ has two visible rational points.
Over the field $\Q(\sqrt{2})$, $X_{618}$ maps to  the elliptic curve
\[
E \colon  y^2 = x^3 + (\sqrt{2} + 1)x^2 + (-3\sqrt{2} -5)x + (-2\sqrt{2} - 3)
\]
which has rank $0$ over $\Q(\sqrt{2})$ and has four $\Q(\sqrt{2})$-rational points, two of which lift to rational points of $X_{618}$.

We found this cover by computing  $\Aut X_{618, \Q(\sqrt{2})}$ (which has order 8) and computing $E$ as the quotient of $X_{618, \Q(\sqrt{2})}$ by one of these automorphisms. (See Subsection \ref{ssec:automorphismGroups} for a description of this computation.)

\subsection{Analysis of $X_{619}$}
\label{ssec:analysis-x_619}
The above techniques do not work on $X_{619}$; its Jacobian has (probable) analytic rank 3 and, while it admits an \'etale double cover $D$, a twist of $D$ has rational points and associated Prym variety of rank 2. The curve $D_{\delta}$ has the equation
\begin{align*}
  \delta r^{2} &= -2705u^{2} + 1681v^{2} - 1967vw + 2048w^{2}\\
  \delta rs &= 73u^{2} - 41v^{2} + 64vw - 64w^{2}\\
  \delta s^{2} &= -2u^{2} + v^{2} - 2vw + 2w^{2}.
\end{align*}

A bit of work reduces this to an elliptic Chabauty computation. Over
the quartic field $K = \Q(a)$, where $a = \sqrt{2 + \sqrt{2}})$, any quadratic twist
$D_{\delta}$ of $D$ has automorphism group $D_8 \times \Z/2\Z$. Let $H$ be the subgroup  $\langle \iota_1, \iota_2 \rangle$, where
$\iota_{1} : D_{\delta} \to D_{\delta}$ is given by $\iota_{1}(u : v : w : r: s) = (u : -v : -w : r : s)$ and
\begin{align*}
  \iota_{2}(u : v : w : r : s) &= (u : \frac{\sqrt{2}}{2} v - w : -\frac{1}{2} v - \frac{\sqrt{2}}{2} w : \\
  & \frac{1}{18} (-73a^{3} + 228a) r + \frac{1}{18} (-2624a^{3} + 8529a) s : \frac{1}{9} (a^{3} - 3a)r + \frac{1}{18} (73a^{3} - 228a) s).
\end{align*}

The twist $D_{-2}$ has no $\Q_2$ points. When $\delta = 1$ or $2$, the quotient $D_{\delta}/H$ is isomorphic to the elliptic curve
\[
E_{+} \colon \delta y^2 = x^3 + (a^3 + 1)x^2 + (194a^3 + 153a^2 - 660a - 509)x + (-1815a^3 - 1389a^2 + 6202a + 4747)
\]
 and the quotient $D_{-1}/H$ is isomorphic to the elliptic curve
\[
E_{-} \colon \delta y^2 = x^3 + (a^3 + a^2 + a + 1)x^2 + (4a^3 + 8a^2 + 6a - 11)x + (-3a^3 + 29a^2 + 11a - 27).
\]
 The quotient of $D_{\delta}$ by $\Aut D_{\delta}$ is $\P^1$; the quotient map $\phi_{\delta}\colon D_{\delta} \to \P^1$ is defined over $\Q$ and factors through the map $D_{\delta} \to E_{\pm}$:
\[
\xymatrix{
D_{\delta} \ar[rd]\ar[dd]_{\phi_{\delta}} &  \\
& E_{\pm} \ar[dl]^{\psi_{\delta}}\\
\P^1& \\
}
\]
We are thus in the situation of elliptic Chabauty -- by construction, any $K$-point of $E_{\pm}$  that is the image of a $\Q$-point of $D_{\delta}$ maps to $\P^1(\Q)$ under $\psi_{\delta}$,
$K$ has degree 4 and $E_{\pm} (K)$ has rank 2.  Magma computes that the only
$K$-rational points of $E$ that map to $\P^1(\Q)$ are the known ones coming from $D_{\delta}$.

It takes a bit of work to compute explicitly the  map $\psi_{\delta} \colon E_{\pm} \to \P^1$. The group $H$ is not normal, so $\psi_{\delta}$ is not given by the quotient of a group of automorphisms. We proceed by brute force. We know the degree of $\psi_{\delta}$ and thus the general form of its equations (by Lemma~\ref{ellipticfunc}). We construct points on $D_{\delta}$ over various number fields; we can map them on the one hand to $E_{\pm}$ and on the other hand to $\P^1$, giving a collection of pairs $(P \in E_{\pm}(\Kbar),\psi_{\delta}(P))$. Sufficiently many such pairs will allow us to compute equations for $\psi_{\delta}$.
\\

See the transcript of computations for code verifying these claims. We find
that there are six rational points on $X_{619}$. Two of these are cusps,
two of these are CM points, corresponding to $j = 16581375$ (CM curves with discriminant $-28$), and two of these
correspond to $j = \frac{857985^{3}}{62^{8}}$. Three other
curves in our list are isomorphic to $X_{619}$. One of these, $X_{649}$
also has non-CM rational points corresponding to $j = \frac{919425^{3}}{496^{4}}$.

\subsection{Analysis of $X_{628}$}

The (probable) analytic rank of the Jacobian of $X_{628}$ is 3, ruling out the possibility of a direct Chabauty argument. While it admits an \'etale double cover,  the Prym variety associated to each twist has rank 1 and Chabauty on the double cover is thus possible but tedious to implement. Alternatively, each \'etale double cover maps to a rank 0 elliptic curve. This map is not explicit and would require a moderate amount of ad hoc work to exploit.

Instead, we exploit the nice model $y^4 = 4xz(x^2 - 2z^2)$ of this curve via the following direct argument. (This is equivalent to \'etale descent, but the simplicity of the model motivates a direct presentation.) An elementary argument shows that, for $xyz \neq 0$, there exist integers $u,v,w$ such that either $x = \pm u^4, z = \pm 4v^4$, and $x^2 - 2z^2 = \pm w^4$, giving $u^8 - 32v^8 = \pm w^4$, or that $x = \pm 2u^4, z = \pm v^4$, and $x^2 - 2z^2 = \pm 2w^4$, giving $2u^8 - v^8 = \pm w^4$. It follows from \cite[Exercise 6.24, Proposition 6.5.4]{cohen:NTI} that the only solution is to the latter equation with $u = v = w = 1$. It
follows that the only points on $y^{4} = 4xz(x^{2} - 2z^{2})$ are
$(0 : 0 : 1)$, $(1 : 0 : 0)$, $(2 : -2 : 1)$ and $(2 : 2 : 1)$.

\subsection{Analysis of  $X_{641}$ and $X_{650}$}
Each of $X_{641}$ and $X_{650}$ have Jacobians of (probable) analytic rank 3, but admit various \'etale double covers.
Each double cover has a twist with local points and such that the associated Prym variety has rank 1. This suggests a Chabauty argument via the Prym, but the details of such an implementation would be complicated. Instead we exploit the nice plane quartic models of these curves.

$X_{641}$ has an affine model  $(x^2 - 2y^2 - 2z^2)^2 = (y^2 - 2yz + 3z^2)(y^2 + z^2)$ and thus admits an \'etale double cover by the curve
\[
\begin{array}{rrcl}
D_{{ \delta}}\colon & y^2 - 2yz + 3z^2 &= & \, { \delta}u^2\\
                          & x^2 - 2y^2 - 2z^2 &= & \, { \delta}uv\\
                         &  y^2 + z^2 &= & \, { \delta}v^2.
\end{array}
\]
The only twist with $2$-adic points is $\delta = 1$. The quotient by the automorphism $[x:y:z:u:v] \mapsto [-x:y:z:-u:-v]$ is the genus 3 hyperelliptic curve $y^2 = -x^8 + 8x^6 - 20x^4 + 16x^2 - 2$. This curve is an unramified double cover of $H\colon y^2 = -x^5 + 8x^4 - 20x^3 + 16x^2 - 2x$. The Jacobian of $H$ has rank $1$, and Chabauty successfully determines the rational points on $H$; computing the preimages of these points on $D$ allows us to conclude that only rational points on $X_{641}$ are the known ones.

Similarly, $X_{650}$ has a model  $y^4 = (x^2 - 2xz - z^2)(x^2 + z^2)$ and thus admits an \'etale double cover by the curve
\[
\begin{array}{rrcl}
D_{{ \delta}}\colon &x^2 - 2xz - z^2 &= & \, { \delta}u^2\\
                          & y^2 &= & \, { \delta}uv\\
                         & x^2 + z^2&= & \, { \delta}v^2
\end{array}
\]
The only twist with $2$-adic points is $\delta = 1$. This genus 5 curve has four automorphisms over $\Q$, and the quotient of $D_1$ by one of the involutions is the genus 3 hyperelliptic curve $y^2 = -x^8 + 2$, which maps to the genus $2$ curve $H\colon y^2 = -x^5 + 2x$. The rank of the Jacobian of $H$ is 1, and Chabauty again proves that the only rational points on  $X_{650}$ are the known points.



\section{Analysis of Rational Points - Genus 5 and 7}

There are 20 genus 5 curves (at most 10 isomorphism classes) and 4 genus 7 curves. 
The genus 5 curves $X_{686}$ and $X_{689}$ are handled in an ad hoc manner by explicit \'etale descent. The remaining genus 5 curves and all of the genus 7 curves are handled by the modular double cover method (see Subsection \ref{ssec:non-explicit-modular}) or are isomorphic to one of $X_{686}$ or $X_{689}$.

\subsection{Analysis of $X_{689}$}

The curve $X_{689}$ has a model
\begin{align*}
 X_{672}\colon  y^2 &  = x^3 + x^2 - 3x + 1 \\
  w^2 & = 2(y^2 + y(-x+1))(x^2 - 2x - 1)
\end{align*}
The curve $D_{\delta}$
\begin{align*}
  y^2 &  =  x^3 + x^2 - 3x + 1 \\
  \delta w_1^2 & = (x^2 - 2x - 1)    \\
  \delta w_2^2 & = 2(y^2 + y(-x+1))
\end{align*}
is an \'etale double cover of $X_{689}$. (Magma computes that $g(D) = 9$, so this follows from Riemann-Hurwitz.) The cover is unramified outside of $2$, so every rational point on $X_{689}$ lifts to a rational point on $D_{\delta}$ for some $\delta \in \{\pm 1, \pm 2\}$.
The curve $D_{\delta}$ maps to the curve $H_{\delta}$ given by
\begin{align*}
  y^2 - (x^3 + x^2 - 3x + 1) &  = 0\\
  \delta w_1^2 - (x^2 - 2x - 1)    & = 0
\end{align*}
which Magma computes is a genus 3 hyperelliptic curve.
Each of these hyperelliptic curves has Jacobian of rank 1 or 2, with four visibile automorphisms. Taking the quotient by a non-hyperelliptic involution gives a genus 2 hyperelliptic curve, the Jacobians of which have rank at most 1; Chabauty applied to the genus 2 curves thus proves that the only rational points on  $X_{672}$ are the known points.

See the transcript of computations for Magma code verifying these claims.

\subsection{Analysis of $X_{686}$}

Similarly, the curve $X_{686}$ has a model
\begin{align*}
 X_{686}\colon  y^2 &  = x^3 + x^2 - 3x + 1 \\
  w^2 & = 2(y^2 -y(-x+1))(x^2 - 2x - 1)
\end{align*}
and \'etale double covers $D_{\delta} \to X_{686}$ from the curves
\begin{align*}
  y^2 &  =  x^3 + x^2 - 3x + 1 \\
  \delta w_1^2 & = x^2 - 2x - 1    \\
  \delta w_2^2 & = 2(y^2 - y(-x+1)).
\end{align*}
The curve $D_{\delta}$ maps to the genus 3 hyperelliptic curve $H_{\delta}$ given by
\begin{align*}
  y^2 - (x^3 + x^2 - 3x + 1) &  = 0\\
  \delta w_1^2 - (x^2 - 2x - 1)    & = 0.
\end{align*}
These are the same curves as in the analysis of $X_{689}$, and we conclude in the same way that the only rational points on  $X_{686}$ are the known points.

\subsection{Non-explicit, modular double covers}
\label{ssec:non-explicit-modular}

The remaining genus 5  curves and the genus 7 curves are inaccessible via other methods and will be handled by the modular double cover method described in subsection \ref{ssec:etale-descent-via}. We describe this method in more detail here.

Let $S = \{ 1, 2, -1, -2 \}$ and for $\delta \in S$ define
$\chi_{\delta}$ to be the Kronecker character associated to $\Q(\sqrt{\delta})$.
Suppose that $X$ is one of these $20$ such curves, with corresponding subgroup
$H$. In each case, we can find four index $2$ subgroups $K_{\delta}$ with
$\delta \in S$ so that for all $g \in K_{\delta}$,
\[
  g \in K_{1} \text{ if and only if } \chi_{\delta}(\det g) = 1.
\]
Moreover, if the genus of $X$ is $g$, the genus of each $K_{\delta}$ is $2g-1$,
which implies that $X_{K_{\delta}}/X$ is \'etale.

Choose a modular function $h(z)$ for $K_{1}$ so that if $m$ is an element of
the non-identity coset for $K_{1}$ in $H$, then $h | m = -h$. A model for
$X_{K_{1}}$ is then given by $h^{2} = r$, where $r \in \Q(X_{H})$. Moreover,
the condition on elements of $K_{\delta}$ implies that $\sqrt{\delta} h$ is
fixed by the action of $K_{\delta}$ (recall the method of model computations in Section~\ref{sec:comp-equat-x_h}). 
This implies that the curves
$X_{K_{\delta}}$ are the twists (by the elements of $S$) of $K_{1}$, and hence
every rational point on $X_{H}$ lifts to one of the $X_{K_{\delta}}$.
In each case, the $X_{K_{\delta}}$ maps to a curve $X_{n}$ whose model we have
computed that has finitely many rational points (namely a pointless conic,
a pointless genus $1$ curve, or an elliptic curve with rank zero).

Note that the group theory alone provides the properties we need for
the curves $X_{K_{\delta}}$, and we do not construct models for them.

\begin{example}
The curve $X_{695}$ is a genus $5$ curve that has two visible rational points
corresponding to elliptic curves with $j$-invariant $54000$. In this case,
$X_{K_{1}}$ and $X_{K_{-1}}$ map to the rank zero elliptic curve
$X_{285} : y^{2} = x^{3} + x$ (whose two rational
points map to $j = 54000$). The curves $X_{K_{2}}$ and $X_{K_{-2}}$ map to
$X_{283}$, a genus $1$ curve with no $2$-adic points.
\end{example}

See the transcript of computations for further details.

\appendix
\section{Proving the mod $N$ representation is surjective}
\label{App:computing-Image}


Given a Galois extension $K/\Q$ with Galois group $G$, \cite{dokchitsers:frobenius} gives an algorithm that will allow one to determine, for a given unramified prime $p$, the Frobenius conjugacy class
${\rm Frob}_{p}$. Applied to the case $K = \Q(E[N])$, and given initial knowledge that $G$ is a subgroup of some particular $H$ (e.g.~$E$ could arise from a rational point on $X_H$), this gives an algorithm to prove that $\im \rho_{E,N} = H$.

\begin{remark}
When $H = S_n$ or $\GL_2(\F_{\ell})$ this is well understood
(e.g.~in the latter case, if $\ell > 5$ and $G$ contains three elements with particular properties then $G = H$ \cite{Serre:openImage}*{Prop. 19}).
For subgroups of $\GL_2(\F_{\ell})$, \cite{sutherland:imageOfGaloisTalk} recently proved that if two subgroups $H,K$ of $\GL_2(\F_{\ell})$ have the same \emph{signature}, defined to be
\[
s_H := \{(\det A, \tr A, \rank \fix A) : A \in H\},
\]
then $H$ and $K$ are conjugate. (Note that the extra data of $\fix A$ is necessary to distinguish the trivial and order 2 subgroups of $\GL_2(\F_2)$.
Already for $G \subset \GL_2(\Z/\ell^2\Z)$ with $\ell > 2$, the additional data of $\fix  A$ does not suffice -- for instance, the order $\ell$ subgroups generated by
\[
\left[ \begin{matrix} 1 - \ell & \ell \\ 0 & 1 + \ell \end{matrix} \right] \text{ and }
\left[ \begin{matrix} 1 -\ell & -\ell \\ 0 & 1+\ell \end{matrix} \right]
\]
have the same signature.)
\end{remark}

\begin{remark}
 It is in principle completely straight-forward to provably determine the image of $\rho_{E,n}$. Indeed, Magma can compute, for any $n$, the corresponding division polynomial, and compute the Galois group of the corresponding field. In practice though, as the degree of $\Q(E[n])$ grows, a direct computation of the Galois group using Magma's built in commands quickly becomes infeasible.
\end{remark}

We now describe the algorithm. Suppose that $K$ is the splitting field of
\[
  F(x) = \prod_{i=1}^{n} (x - a_{i}).
\]
Given some fixed polynomial $h$ and a conjugacy class $C \subseteq G$, construct
the resolvent polynomial
\[
  \Gamma_{C}(X) = \prod_{\sigma \in C} \left(X - \sum_{i=1}^{n} h(a_{i}) \sigma(a_{i})\right).
\]
Theorem 5.3 of \cite{dokchitsers:frobenius} states the following (specializing to extensions of $\Q$).

\begin{thmnonum}
Assume the notation above.
\begin{enumerate}
\item For each conjugacy class $C \subseteq G$, $\Gamma_{C}(X)$ has coefficients
in $\Q$.\\
\item If $p$ is a prime that does not divide the denominators of $F(x)$,
$h(x)$ and the resolvents of $\Gamma_{C}$ and $\Gamma_{C'}$ for different
$C$ and $C'$, then
\[
  {\rm Frob}_{p} = C \iff \Gamma_{C}\left({\rm Tr}_{\frac{\F_{p}[x]}{F(x)} /
    \F_{p}} (h(x) x^{p})\right) \equiv 0 \pmod{p}.
\]
\end{enumerate}
\end{thmnonum}

We wish to apply this theorem in the case that $G = H$ and when the Galois group of $K/\Q$ may not necessarily be $G$. An examination of the proof shows
that the theorem remains true even if $\Gal(K/\Q)$ is a proper subgroup of $G$.

Our setup is the following. Suppose that $E/\Q$ is an elliptic curve with
a model chosen that has integer coefficients. Suppose also that we know,
a priori, that the image of the mod $N$ Galois representation
is contained in $H \subseteq \GL_{2}(\Z/N\Z)$. The following algorithm gives
a method to prove that the mod $N$ image is equal to $H$. Define
\begin{align*}
  s_{1}(N) &= \begin{cases} 4 & \text{ if } N = 2\\
  p & \text{ if } N > 2 \text{ is a power of the prime } p\\
  1 & \text{ otherwise,}
  \end{cases}\\
  s_{2}(N) &= \begin{cases} 8 & \text{ if } N = 2\\
  9 & \text{ if } N = 3\\
  p & \text{ if } N > 3 \text{ is a power of the prime } p\\
  1 & \text{ otherwise.}
  \end{cases}
\end{align*}

\begin{enumerate}
\item We fix an isomorphism $\phi \colon (\Z/N\Z)^{2} \to E[N]$ and pre-compute decimal
  expansions of $f(P) = s_{1}(N) x(P) + s_{2}(N) y(P)$ for all torsion points of $P$ of order $N$
  on $E$. By Theorem VIII.7.1 of \cite{Silverman:AEC}, these
  numbers are algebraic integers.
\item The action of Galois on the numbers $s_{1}(N) x(P) + s_{2}(N) y(P)$ is given
by some conjugate of $H$. We attempt to identify a unique conjugate of $H$ in $\GL_{2}(\Z/N\Z)$ that
gives this action. We do this by numerically computing
\[
  \sum_{k \in K} f(\phi(k(1,0))) f(\phi(k(0,1))) + f(\phi(k(1,0))) f(\phi(k(1,1))) + f(\phi(k(0,1))) f(\phi(k(1,1)))
\]
for each conjugate $K$ of $H$ inside $\GL_{2}(\Z/N\Z)$. If the image of the mod $N$ representation is contained in $K$, then the sum above will be an integer.
\item We compute the polynomial $F(x)$ with integer coefficients
whose roots are the numbers $f(P) = s_{1}(N) x(P) + s_{2}(N) y(P)$. This polynomial is computed numerically. Knowing the size of the numbers $f(P)$,
we verify that enough decimal precision is used to be able to round the coefficients of $F(x)$ to the nearest integer and obtain the correct result.
\item We compute the resolvent polynomials for all of the conjugacy
classes of $H$ and check that these have no common factor. (In practice, we use $h(x) = x^{3}$ to construct these polynomials. We use a smaller decimal precision for the resolvent polynomials and again check that we can round the coefficients to the nearest integer to obtain the correct result.)
\item Using the resolvent polynomials, we compute the conjugacy class of $\rho_{E,N}({\rm Frob}_{p}) \subseteq H$ for lots of different primes $p$.
\item We enumerate the maximal subgroups of $H$ and determine which conjugacy classes they intersect. We check to see if the conjugacy classes
found in the previous step all lie in some proper maximal subgroup of $H$. If not, then the image of $\rho_{E,N}$ is equal to $H$.
\end{enumerate}

Note that it is not possible for a maximal subgroup $M \subseteq H$ to intersect all of the conjugacy classes of $H$.

\begin{example}
  Let $E : y^{2} = x^{3} + x^{2} - 28x + 48$. This elliptic curve has
  $j$-invariant $78608$, which corresponds to a non-CM rational point
  on $X_{556}$, and hence the $2$-adic image for $E$ is contained in
  $H_{556}$, an index $96$ subgroup of $\GL_{2}(\Z_{2})$ that contains
  $\Gamma(16)$. We must show that the $2$-adic image equals $H_{556}$.
  Every maximal subgroup of $H_{556}$ also contains $\Gamma(16)$, so
  it suffices to compute the image of the mod $16$ Galois
  representation attached to $E$. To do this, we fix an isomorphism
  $E[16] \cong (\Z/16\Z)^{2}$, and precompute decimal expansions of $2
  x(P) + 2 y(P)$ for all $P \in E[16]$, using $1000$ digits of decimal
  precision. There are $24$ conjugates of $H_{556}$ in
  $\GL_{2}(\Z_{2})$, and we find that the expression in step 2 above
  is an integer only for one of the conjugates of $H_{556}$.

The image of $H_{556}$ under the map $\GL_{2}(\Z_{2}) \to
\GL_{2}(\Z/16\Z)$ has $46$ conjugates classes, and we compute the
polynomial $F(x)$ whose roots are the $192$ numbers $2 x(P) + 2
y(P)$. Knowing the sizes of the roots, we can see that no coefficient
of $F(x)$ could be larger than $10^{291}$, and so $1000$ digits of decimal
precision is enough to correctly recover $F(x)$.

We then compute the resolvent polynomials for the $46$
conjugacy classes (using 500 digits of decimal precision). Then, for
each prime $p \leq 30000$, we compute ${\rm
  Tr}_{\frac{\F_{p}[x]}{(F(x))}/\F_{p}}(x^{p+3})$ and check which
resolvent polynomial has this number as a root in $\F_{p}$. Using
this, we can determine which conjugacy class is the image of ${\rm
  Frob}_{p}$. We find that all $46$ conjugacy classes are in the image
of ${\rm Frob}_{p}$ for some $p$. (For example, the smallest prime $p$
which splits completely in $\Q(E[16])$ is $p = 5441$.) As a
consequence the image of the mod $16$ Galois representation of $E$ is
$H_{556}$.
\end{example}

\bibliography{master}
\bibliographystyle{alpha}

\end{document}